\documentclass[reqno,12pt,a4paper]{amsart}
\usepackage{amsmath,amssymb,amsthm}
\usepackage{bm,bbm,comment,mathtools}
\usepackage{graphicx}
\newlength{\margins}
\usepackage[top=\margins+10truemm,bottom=\margins,left=\margins,right=\margins]{geometry}
\usepackage{hyperref}

\allowdisplaybreaks
\makeatletter
\@namedef{subjclassname@2020}{%
\textup{2020} Mathematics Subject Classification}
\makeatother
\thanks{This work was supported by the Japan Science and Technology Agency (JST) SPRING Program,
Grant Number JPMJSP2110, and by JSPS KAKENHI Grant Number JP26KJ1378.}
\author[N. Takeda]{Nobuki TAKEDA}

\address{Department of Mathematics, Graduate School of Science, Kyoto University, Kyoto 606-8502, Japan}
\email{takeda.nobuki.z04@kyoto-u.jp}

\theoremstyle{definition}
\newtheorem{definition}{Definition}[section]
\newtheorem{remark}[definition]{Remark}

\theoremstyle{plain}
\newtheorem{proposition}[definition]{Proposition}
\newtheorem{lemma}[definition]{Lemma}
\newtheorem{theorem}[definition]{Theorem}
\newtheorem{corollary}[definition]{Corollary}

\newtheorem*{cond}{Condition (A)}

\DeclareMathOperator{\tr}{Tr}

\DeclareMathOperator{\diag}{diag}
\DeclareMathOperator{\Aut}{Aut}
\DeclareMathOperator{\Her}{Her}
\DeclareMathOperator{\pr}{pr}
\DeclareMathOperator{\Sym}{Sym}
\DeclareMathOperator*{\bigboxtimes}{\boxtimes}

\newcommand{\bsk}{\boldsymbol{k}}
\newcommand{\bsl}{\boldsymbol{l}}

\ExplSyntaxOn
\clist_const:Nn \c_alphabet_clist
  {A,B,C,D,E,F,G,H,I,J,K,L,M,N,O,P,Q,R,S,T,U,V,W,X,Y,Z,
   a,b,c,d,e,f,g,h,i,j,k,l,m,n,o,p,q,r,s,t,u,v,w,x,y,z}
\cs_new:Npn \makealphabet #1 #2
  {\clist_map_inline:Nn \c_alphabet_clist
    {\exp_args:Nc \newcommand {#1##1}{#2{##1}}}}
\makealphabet{bb}{\mathbb}
\makealphabet{bf}{\mathbf}
\makealphabet{cal}{\mathcal}
\makealphabet{frak}{\mathfrak}
\ExplSyntaxOff

\newcommand{\gm}{\mathbb{G}_m}
\newcommand{\GL}{\mathrm{GL}}

\newcommand{\gu}{\mathrm{GU}}
\newcommand{\M}{\mathrm{M}}
\newcommand{\rmu}{\mathrm{U}}
\newcommand{\SL}{\mathrm{SL}}

\newcommand{\dkl}{\bbD_{\bsk,\bsl}}
\newcommand{\kv}{\bsk_v}
\newcommand{\lv}{\bsl_v}
\newcommand{\tilg}{\widetilde{G}}
\newcommand{\tilgam}{\widetilde{\Gamma}}
\newcommand{\tilk}{\widetilde{\frakK}}

\newcommand{\va}{v\in\bfa}

\providecommand{\abs}[1]{\left\lvert #1\right\rvert}
\providecommand{\norm}[1]{\left\lvert #1\right\rvert}
\providecommand{\adj}[1]{{#1^{*}}}

\providecommand{\hus}[1]{\mathfrak{H}_{#1}}

\renewcommand{\Im}{\mathrm{Im}}
\renewcommand{\Re}{\mathrm{Re}}

\numberwithin{equation}{section}
\begin{document}
\title[Congruences of Hermitian Klingen-Eisenstein series]{Congruences between Klingen-Eisenstein series and cusp forms on \(\rmu_{n,n}\)}
\date{}
\begin{abstract}
In this paper, we study congruences of Hecke eigenvalues between Hermitian Klingen--Eisenstein series and cusp forms on the unitary group \(\rmu_{n,n}\) defined over the rational number field \(\bbQ\).
We also prove the rationality of the space of Hermitian automorphic forms and the integrality of their Hecke eigenvalues.
\end{abstract}

\maketitle
\tableofcontents

\section{Introduction}

Studying congruences between Eisenstein series and cusp forms is fundamental to the arithmetic theory of automorphic forms.
These congruences are not merely arithmetic curiosities but are deeply intertwined with the special values of \(L\)-functions and have far-reaching applications in Iwasawa theory.
The existence of a congruence often implies that the \(p\)-adic properties of an \(L\)-value control the arithmetic of the associated Galois representations, a philosophy central to the proof of the Iwasawa Main Conjecture.

A foundational instance of this principle is the work of Skinner and Urban \cite{SkinnerUrban2014Iwasawa}, who established the Main Conjecture for \(\GL_2\).
Their strategy involved constructing congruences between Eisenstein series and cusp forms on the unitary similitude group \(\gu(2,2)\).

Analytic methods for constructing congruences using the pullback formula have been extensively developed, primarily in the context of symplectic groups \(\mathrm{Sp}(n)\).
Recently, Katsurada \cite{katsurada2008congruence} and Katsurada-Mizumoto \cite{Katsurada2012congruences} established a refined criterion for the existence of congruences between Klingen-Eisenstein series and cusp forms on \(\mathrm{Sp}(n)\).
Their method utilizes the pullback formula to explicitly relate the Petersson inner product of a cusp form and a restricted Eisenstein series to special values of standard \(L\)-functions.

In this paper, we extend these powerful techniques to the case of (vector-valued) Hermitian automorphic forms on the unitary group \(\rmu_{n,n}\) defined over the rational number field \(\bbQ\).
The primary objective of this work is to establish a precise criterion for the existence of congruences between Hermitian Klingen-Eisenstein series and Hecke cusp forms on \(\rmu_{n,n}\).

Specifically, let \(r < n\) be a positive integer.
We consider a Hecke cusp form \(f\) on the smaller unitary group \(\rmu_{r,r}\) and its associated Klingen-Eisenstein series \([f]_\mu^n\) (defined in Section~\ref{subsec:Eisenstein}) on \(\rmu_{n,n}\).
The main result of this paper (Theorem \ref{thm:main}) asserts that if the algebraic part of the special value of the standard \(L\)-function associated with \(f\), denoted by \(L(s, f, \mathrm{St})\), is divisible by a prime ideal \(\frakp\) (under appropriate normalization and conditions on the prime \(p\)), then there exists a Hecke cusp form \(F\) on \(\rmu_{n,n}\) such that
\[
  F \equiv_{ev} [f]_{\mu}^{n} \pmod{\frakP},
\]
where \(\frakP\) is a prime ideal of the compositum of the Hecke fields lying above \(\frakp\), and \(\equiv_{ev}\) denotes congruence of eigenvalues for all Hecke operators in the spherical Hecke algebra introduced in Section~\ref{sec:hecke-operators}.

To prove this theorem, we first utilize the vector-valued differential operators \(\dkl\) constructed in our previous work \cite{Takeda2025pullback}, which are designed to ensure that the pullback of an automorphic form remains automorphic on the subgroup.
By applying these operators to the Eisenstein series on \(\rmu_{n,n}\) and restricting it to \(\rmu_{n_1,n_1} \times \rmu_{n_2,n_2}\), we derive a pullback formula (Theorem \ref{thm:pullback}).
At the same time, we establish the necessary arithmetic framework by defining an integral structure on the space of automorphic forms and proving the rationality of the space of Hermitian automorphic forms (Theorem \ref{thm:integrality_automorphic}).

Finally, we apply the strategy of Katsurada-Mizumoto \cite{Katsurada2012congruences} to the Hermitian setting by analyzing the denominators of the Fourier coefficients of the Klingen-Eisenstein series \([f]_\mu^n\).

\medskip
We now describe our main result (Theorem~\ref{thm:main}) in more detail.

Let \(n\) be a positive integer and let \(\mu\) be a positive even integer with \(\mu > 2n\).
Let \((\bsk,\bsl)\) be a pair of dominant integral weights whose lengths satisfy \(\ell(\bsk), \ell(\bsl) \le n\).
Let \(f\) be a Hecke cusp form on the unitary group \(\rmu_{r,r}\) of signature \((r,r)\), and denote by \([f]_\mu^n\) the Klingen-Eisenstein series on \(\rmu_{n,n}\).
We write \(\bbQ(f)\) for the Hecke field generated by the eigenvalues of \(f\).

To formulate the congruence criterion, we introduce a constant
\[
  \calC_{n,\mu}(f)
  =
  \frac{\bbL_F(n,\mu)}{\bbL_F(2r,\mu)}\,
  \bbL\bigl((\mu+1)/2-r,f\bigr).
\]
(See Section \ref{sec:integrality} for the precise definitions and normalization.)

Let \(\frakp\) be a prime ideal of \(\bbQ(f)\).
Assume that \(\frakp\) satisfies
\[
  v_\frakp\!\left(
  \overline{\calC_{2n,\mu}(f)}\,
  A_{[f]_\mu^n}(\gamma_0,S_0)\,
  \overline{A_{[f]_\mu^n}(\gamma_0,S_0)}
  \right)
  < 0,
\]
where \(v_\frakp\) denotes the \(\frakp\)-adic valuation, and \(A_{[f]_\mu^n}(\ast,S_0)\) is the normalized \(S_0\)-th Fourier coefficient of \([f]_\mu^n\) defined in~\eqref{eq:arithmetic_fourier}.
We further impose several natural technical conditions on \(\frakp\) ensuring integrality.

Then there exists a Hecke cusp form \(F\) on \(\rmu_{n,n}\) such that
\[
  F \equiv_{ev} [f]_\mu^n \pmod{\frakP}
\]
for some prime ideal \(\frakP\) of \(\bbQ(f,F)\) lying above \(\frakp\).
Here the notation \(F \equiv_{ev} G \pmod{\frakP}\) means that the Hecke eigenvalues of \(F\) and \(G\) are congruent modulo \(\frakP\) for all Hecke operators in the spherical Hecke algebra.

Moreover, under a stronger condition on the \(\frakp\)-adic valuation of the above algebraic quantity, we obtain a refined congruence modulo a higher power of \(\frakP\).

This paper is organized as follows:
Section 2 provides the preliminary definitions of Hermitian automorphic forms on \(\rmu_{n,n}\) and their Fourier expansions.
Section 3 is devoted to the local theory of Hecke algebras, describing the structure at inert, ramified, and split places.
Section 4 introduces the differential operators and states the pullback formula (Theorem \ref{thm:pullback}).
Section 5 discusses the arithmetic properties and rationality of the space of automorphic forms (Theorem \ref{thm:integrality_automorphic}).
Section 6 contains the proof of the main result.
Section 7 provides explicit numerical examples for \(n=2\) and \(n=3\).
\subsection*{Acknowledgements}
The author is grateful to T. Ikeda for his guidance and support as the author's doctoral advisor, and to H. Katsurada for valuable comments.
This work was supported by the Japan Science and Technology Agency (JST) SPRING Program,
Grant Number JPMJSP2110,
and by JSPS KAKENHI Grant Number JP26KJ1378.
\subsection*{Notation}
We denote by \(\M_{m,n}(R)\) the set of \(m\times n\) matrices with entries in \(R\).
In particular, we put \(\M_n(R) := \M_{n,n}(R)\).
Let \(I_n\) be the identity element of \(\M_n(R)\) and \(e_{ij}\)
the matrix with \(1\) at the \((i,j)\)-th entry and \(0\) elsewhere.
Let \(\det(X)\) denote the determinant of \(X\) and \(\tr(X)\) the trace of \(X\).
For a square matrix \(X \in \M_n(R)\), let \({}^tX\) denote the transpose of \(X\).
Let \(\GL_n(R) \subset \M_n(R)\) be a general linear group of degree \(n\).

Let \(K\) be a quadratic extension field of \(K_0\), and let \(\rho\) be the non-trivial automorphism of \(K\) over \(K_0\).
We often put \(\overline{x}=\rho(x)\) for \(x\in K\).
We put \(\overline{X}=(\overline{x_{ij}})\) and \(\adj{X}= {}^t\overline{X}\) for \(X=(x_{ij})\in \M_{m,n}(K)\).
Let \(L\) be an algebraic number field, and let \(\frakp\) be a prime ideal of \(L\).
We denote by \(L_\frakp\) the \(\frakp\)-adic completion of \(L\)
and by \(\calO_L\) (resp. \(\calO_\frakp\)) the ring of integers of \(L\) (resp. of \(L_\frakp\)).
Let \(v_\frakp\) be the additive valuation of \(L_\frakp\), normalized by \(v_\frakp(\varpi_\frakp)=1\), where \(\varpi_\frakp\) is a uniformizer of \(L_\frakp\).
We put \(\calO_{L,\frakp}=L\cap\calO_\frakp\), the localization of \(\calO_L\) at \(\frakp\).
Let \(\Her_n(\bbC)\subset \M_n(\bbC)\) be the set of Hermitian matrices.
For an element \(X \in \Her_n(\bbC)\),
we denote by \(X>0\) (resp. \(X \geq 0\)) that \(X\) is a positive definite matrix (resp. a non-negative definite matrix).
For a subset \(S \subset \Her_n(\bbC)\), we denote by \(S_{>0}\) (resp. \(S_{\geq 0}\))
the subset of positive definite (resp. non-negative definite) matrices in \(S\).
If a group \(G\) acts on a set \(V\), then we denote by \(V^G\) the \(G\)-invariant
subspace of \(V\).

Let \({\det}^k\) be the one-dimensional representation of \(\GL_n(\bbC)\) given by the \(k\)-th power of the determinant,
and let \(\Sym^l\) be the \(l\)-th symmetric power representation of \(\GL_n(\bbC)\).
For a representation \((\rho,V)\), we denote by \((\rho^*, V^*)\)
the contragredient representation of \((\rho, V)\).

We denote by
\((x)^{(r)} = x(x+1)\cdots(x+r-1)\) (resp. \((x)_{(r)} = x(x-1)\cdots(x-r+1)\))
the ascending (resp. descending) Pochhammer symbol.

\section{Hermitian automorphic forms}\label{sec:hermitian}
Let \(K\) be a quadratic imaginary extension of a totally real field \(F\), and let \(\bfa\) (resp. \(\bfh\)) denote the set of infinite (resp. finite) places of \(F\).
We fix a CM type \(\Sigma_K\) for \(K\) (i.e.
\(\Sigma_K\) contains a choice of exactly one representative from each pair of complex conjugate embeddings of \(K\)).
We often canonically identify the CM type \(\Sigma_K\) with \(\bfa\), and we denote by \(\sigma_v\) the embedding \(K \hookrightarrow \bbC\) corresponding to \(v \in \bfa\).

We put \(m=m_{F}:=\#\bfa=[F:\bbQ]\).
We put \(K_v=\prod_{w|v} K_w\) and \(\calO_{K_v}=\prod_{w|v}\calO_{K_w}\) for a place \(v\) of \(F\).
Let \(\bbA=\bbA_F\) be the adele ring of \(F\), and \(\bbA_{0}\),
\(\bbA_{\infty}\) the finite and infinite parts of \(\bbA\), respectively.

We put \(J_n=
\begin{pmatrix}O_n & I_n \\ -I_n & O_n \\
\end{pmatrix}\).
The unitary group \(\rmu_{n,n}\) and the unitary similitude group \(\gu_{n,n}\) are algebraic groups defined over the totally real field \(F\).
For any \(F\)-algebra \(R\), their \(R\)-points are given by
\begin{align*}
  \rmu_{n,n}(R) & =\{g\in \GL_{2n}(K\otimes_{F}R) \mid \adj{g}J_ng=J_n\},                              \\
  \gu_{n,n}(R)  & =\{g\in \GL_{2n}(K\otimes_{F}R) \mid \adj{g}J_ng=\nu(g)J_n, \ \nu(g)\in  R^\times\}.
\end{align*}
We call \(\nu(g)\) the similitude factor of \(g\).
If we put
\[\pi(g)=
  \begin{pmatrix}\nu(g)I_n & 0_n \\0_n&I_n
  \end{pmatrix}^{\!-1}\hspace{-5pt}g\in \rmu_{n,n}\]
for \(g\in \gu_{n,n}\),
we have the isomorphism
\begin{eqnarray}
  \gu_{n,n}&\cong&\rmu_{n,n}\rtimes\gm,\nonumber\\
  g&\mapsto&\left(\pi(g),\ \nu(g)\right),\label{eq:gu_isom}.
\end{eqnarray}

We also define other unitary groups \(\rmu(n,n)\), \(\gu(n,n)\), \(\gu^+(n,n)\),
and \(\rmu(n)\) by
\begin{align*}
  \rmu(n)    & = \{ g \in \GL_n(\bbC) \mid \adj{g} g = I_n \},                                              \\
  \rmu(n,n)  & = \{ g \in \GL_{2n}(\bbC) \mid \adj{g} J_n g = J_n \},                                       \\
  \gu(n,n)   & = \{ g \in \GL_{2n}(\bbC) \mid \adj{g} J_n g = \nu(g) J_n, \ \nu(g) \in \bbR^\times \},      \\
  \gu^+(n,n) & = \{ g \in \GL_{2n}(\bbC) \mid \adj{g} J_n g = \nu(g) J_n, \ \nu(g) \in \bbR^\times_{>0} \}.
\end{align*}

We put \(\tilg_n = \gu_{n,n}(F)\) and
\[
  \tilg^+_n = \{ g \in \gu_{n,n}(F) \mid \nu(g) \in F_{>0} \},
\]
where
\[
  F_{>0} = \{ x \in F \mid \sigma(x) > 0 \text{ for all embeddings } \sigma : F \hookrightarrow \bbR \}.
\]

Let \(\tilg_{n,v} = \gu_{n,n}(F_v)\) for a place \(v\) of \(F\),
\[
  \tilg_n(\bbA) = \gu_{n,n}(\bbA), \quad
  \tilg_{n,0} = \gu_{n,n}(\bbA_0),
\]
\[
  \tilg_{n,\infty} = \gu_{n,n}(F \otimes_\bbQ \bbR) = \prod_{v \in \mathbf{a}} \tilg_{n,v} = \prod_{v \in \mathbf{a}} \gu(n,n).
\]
We also set
\[
  \tilg^+_{n,\infty} = \prod_{v \in \mathbf{a}} \gu^+(n,n), \qquad
  \tilg^+_n(\bbA) = \prod_{v \in \mathbf{h}} \gu_{n,n}(F_v) \times \tilg^+_{n,\infty}.
\]

Similarly, for the unitary group \(\rmu_{n,n}\), we define
\[
  G_n = \rmu_{n,n}(F), \quad G_{n,v} = \rmu_{n,n}(F_v), \quad \text{etc.}
\]

We define \(\frakK_{n,v}\) and \(\tilk_{n,v}\) by
\begin{equation*}
  \frakK_{n,v}=\left\{
  \begin{array}{ll}
    \rmu_{n,n}(\calO_{F_v}) & (v\in\bfh), \\
    \rmu(n)\times \rmu(n)   & (\va),
  \end{array}
  \right.\quad \text{and}\quad
  \tilk_{n,v}=\left\{
  \begin{array}{ll}
    \gu_{n,n}(\calO_{F_v}) & (v\in\bfh), \\
    \rmu(n)\times \rmu(n)  & (\va).
  \end{array}
  \right.
\end{equation*}
Then \(\frakK_{n,v}\) (resp. \(\tilk_{n,v}\)) is isomorphic to a maximal compact subgroup of \(G_{n,v}\) (resp. \(\tilg^+_{n,v}\)) for each place \(\va\).
We fix a maximal compact subgroup of \(G_{n,v}\) (resp. \(\tilg_{n,v}\)),
which is also denoted by \(\frakK_{n,v}\) (resp. \(\tilk_{n,v}\)) by abuse of notation.
We put \(\frakK_{n,0}=\prod_{v\in\bfh}\frakK_{n,v}\), \(\frakK_{n,\infty}=\prod_{\va }\frakK_{n,v}\),
\(\tilk_{n,0}=\prod_{v\in\bfh}\tilk_{n,v}\) and \(\tilk_{n,\infty}=\prod_{\va }\tilk_{n,v}\).

\subsection{As analytic functions on Hermitian symmetric spaces}\label{subsec:analytic}
We have the identification
\begin{align*}
  \M_n(\bbC) & \cong \Her_n(\bbC)\otimes_\bbR \bbC, \\
  Z          & \mapsto \Re(Z) + \sqrt{-1}\,\Im(Z),
\end{align*}
where the Hermitian real and imaginary parts are given by
\begin{align*}
  \Re(Z) & = \tfrac{1}{2}(Z+\adj{Z}),          \\
  \Im(Z) & = \tfrac{1}{2\sqrt{-1}}(Z-\adj{Z}).
\end{align*}

Let \(\hus{n}\) denote the Hermitian upper half space of degree \(n\):
\[
  \hus{n} = \{ Z \in \M_n(\bbC) \mid \Im(Z) > 0 \}.
\]

The group \(\tilg^+_{n,\infty}=\prod_{\va}\gu^+(n,n)\) acts on \(\hus{n}^\bfa\) by
\[
  gZ = \bigl((A_v Z_v + B_v)(C_v Z_v + D_v)^{-1}\bigr)_{\va},
\]
for \(g=
\begin{pmatrix}A_v & B_v \\ C_v & D_v
\end{pmatrix}_{\va}\in \tilg^+_{n,\infty}\) and \(Z=(Z_v)_{\va}\in\hus{n}^\bfa\).
Set \(\bfi_n := (\sqrt{-1}I_n)_{\va}\in \hus{n}^\bfa\).

For \(g=(g_v)_{\va}=
\begin{pmatrix}A_v & B_v \\ C_v & D_v
\end{pmatrix}_{\va}\in \tilg^+_{n,\infty}\) and \(Z=(Z_v)_{\va}\in\hus{n}^\bfa\), define
\[
  \lambda(g,Z) = \bigl(\nu(g_v)^{-1/2}(C_v Z_v + D_v)\bigr)_{\va}, \quad
  \mu(g,Z) = \bigl(\nu(g_v)^{-1/2}(\overline{C_v}{}^t Z_v + \overline{D_v})\bigr)_{\va},
\]
and
\[
  M(g,Z) = (\lambda(g,Z),\, \mu(g,Z)).
\]
For brevity, we write
\[
  \lambda(g) = \lambda(g,\bfi_n), \quad
  \mu(g) = \mu(g,\bfi_n), \quad
  M(g) = M(g,\bfi_n).
\]

Let \((\rho, V )\) be an algebraic representation of
\(\tilk_{n,\infty}^\bbC:=\prod_{\va}(\GL_n(\bbC)\times\GL_n(\bbC))\) on a finite-dimensional complex vector space \(V\),
and take a Hermitian inner product on \(V\) such that
\[\left<\rho(g)v,w\right>=\left<v,\rho(\adj{g})w\right>\]
for any \(g \in \tilk_{n,\infty}^\bbC\).
Let \(\omega: (\bbR_{>0})^\bfa\rightarrow \bbC^\times\) be a continuous character.
For a \(V\)-valued function \(F\) on \(\hus{n}^\bfa\), we put
\[ F|_{(\rho,\omega)}g(Z)=\omega(\nu(g))^{-1}\rho(M(g,Z))^{-1}F(g\langle Z\rangle) \quad (g\in \tilg^+_{n,\infty},\ Z \in \hus{n}^\bfa).\]
We set

\[\tilgam_n=\tilg^+_{n,\infty}\cap \GL_{2n}(\calO_{K}).\]
\begin{definition}
  We say that \(F\) is a (holomorphic, level~1) \emph{Hermitian modular form} of weight \((\rho,V)\) and character \(\omega\)
  if \(F\) is a holomorphic \(V\)-valued function on \(\hus{n}^\bfa\) and \(F|_{(\rho,\omega)}g=F\) for all \(g \in \tilgam_n\).
  (If \(n=1\) and \(F=\bbQ\), another holomorphy condition at the cusps is also needed.)

  We denote by \(M_\rho(\tilgam_n,\omega)\) a complex vector space of all Hermitian modular forms
  of weight \((\rho,V)\), and character \(\omega\).
\end{definition}

We set
\[
  \Lambda_n := \Her_n(K)\cap M_n(\mathfrak d_K^{-1}),
\]
where \(\mathfrak d_K\) denotes the different ideal of \(K\).
We further denote by \({\Lambda_n}_{\ge 0}\) the subset of \(\Lambda_n\) consisting of matrices that are non-negative definite for every embedding \(\sigma\in\Sigma_K\).

Then, a modular form \(F \in M_\rho(\tilgam_n,\omega)\) has the Fourier expansion
\[ F(Z)=\sum_{T\in {\Lambda_n}_{\geq0}}a(F,T)\bfe\left(\sum_{\va}(\tr(T_{\sigma_v} Z_v))\right),\]
where \(a(F,T)\in V\) and \(\bfe(z)=\exp(2\pi\sqrt{-1}z)\).
Here, \(T_{\sigma_v}\) is the image of \(T\in \Her_n(K)\) under the embedding corresponding to \(\sigma_v\in\Sigma_K\).
If \(a(F, T)=0\) unless \(T\) is positive definite,

we say that \(F\) is a (holomorphic) Hermitian cusp form of weight \((\rho,V)\) and character \(\omega\).
We also denote by \(S_\rho(\tilgam_n,\omega)\) a complex vector space of all cusp forms of weight \((\rho,V)\), and character \(\omega\).

Write the variable \(Z=(X_v+\sqrt{-1}Y_v)_{\va}\) on \(\hus{n}^\bfa\) with \(X_v,Y_v \in \Her_n(\bbC)\) for each \(\va\).
We identify \(\Her_n(\bbC)\) with \(\bbR^{n^2}\) and define measures \(dX_v, dY_v\) as the
standard measures on \(\bbR^{n^2}\).
We define a measure \(dZ\) on \(\hus{n}^\bfa\) by
\[dZ= \prod_{\va}dX_v dY_v.\]
For \(F,G \in M_\rho(\tilgam_n,\omega)\), we define the Petersson inner product
by
\[ (F,G)=\int_{D } \left\langle\rho(Y^{1/2},{}^tY^{1/2})F(Z),\rho(Y^{1/2},{}^tY^{1/2})G(Z)\right\rangle(\prod_{\va}\det(Y_v)^{-2n})dZ,\]
where \(Y=(Y_v)_{\va}=\Im(Z)\), \(Y^{1/2}=(Y^{1/2}_v)_{\va}\) is a family of positive definite Hermitian matrices such that
\((Y_v^{1/2})^2=Y_v\), and \(D\) is a Siegel domain on \(\hus{n}^\bfa\)
for \(\tilgam_n\).
This integral converges if either \(F\) or \(G\) is a cusp form.
We call a sequence of non-negative integers \(\bsk = (k_1, k_2, \dots)\) a \emph{dominant integral weight}
if \(k_i \geq k_{i+1}\) for all \(i\) and \(k_i = 0\) for almost all \(i\).
The largest integer \(m\) such that \(k_m \neq 0\) is called the \emph{length} of \(\bsk\) and denoted by \(\ell(\bsk)\).
Dominant integral weights of length at most \(n\) correspond bijectively to the set of irreducible algebraic
representations of \(\GL_n(\bbC)\).
We denote by \((\rho_{n,\bsk}, V_{n,\bsk})\) the irreducible representation of \(\GL_n(\bbC)\)
corresponding to a dominant integral weight \(\bsk\) with \(\ell(\bsk) \leq n\).

In what follows, we will frequently realize representation spaces in terms of bideterminants,
so we briefly recall the relevant construction.

Let \(\bsk = (k_1, \dots, k_n)\) be a dominant integral weight with \(\ell(\bsk) \leq n\).
We realize the corresponding representation space \(V_{n,\bsk}\) of \(\GL_n(\bbC)\)
as a space of bideterminants (cf.~Ibukiyama \cite{ibukiyama2019construction}).

Let \(U = (u_{ij})\) be an \(\ell(\bsk) \times n\) matrix of variables.
For a positive integer \(a \leq \ell(\bsk)\), let \(SI_{n,a}\) denote the set of strictly increasing sequences
\(J = (j_1, \dots, j_a)\) of positive integers not exceeding \(n\).
For each \(J \in SI_{n,a}\), define the submatrix \(U_J\) of \(U\) by
\[
  U_J =
  \begin{pmatrix}
    u_{1,j_1} & \cdots & u_{1,j_a} \\
    \vdots    & \ddots & \vdots    \\
    u_{a,j_1} & \cdots & u_{a,j_a}
  \end{pmatrix}.
\]

A polynomial \(P(U)\) in the entries of \(U\) is called a \emph{bideterminant of weight \(\bsk\)}
if it has the form
\[
  P(U) = \prod_{i=1}^m \prod_{j=1}^{k_i - k_{i+1}} \det U_{J_{ij}},
\]
where \(J_{ij} \in SI_{n,i}\) for each \(j=1,\ldots, k_i - k_{i+1}\).
Here we adopt the convention that the inner product is \(1\) if \(k_i = k_{i+1}\).

Let \(BD_{\bsk}\) denote the set of all bideterminants of weight \(\bsk\).
For a commutative ring \(R\) and an \(R\)-algebra \(S\), let \(S[U]_{\bsk}\) denote the \(R\)-module of all
\(S\)-linear combinations of elements in \(BD_{\bsk}\).

The group \(\GL_n(\bbC)\) acts on \(\bbC[U]_{\bsk}\) via
\[
  (g, P(U)) \mapsto P(U g), \quad g \in \GL_n(\bbC), \, P(U) \in \bbC[U]_{\bsk},
\]
and under this action, \(\bbC[U]_{\bsk}\) provides a concrete realization of the
representation space \(V_{n,\bsk}\).

We define an inner product on \(\bbC[U]_{\bsk}\) by
\[
  \langle P(U), Q(U) \rangle
  = P(\partial_U)\,\overline{Q}(U)\big|_{U=0},
\]
where \(P(\partial_U)\) denotes the differential operator obtained by replacing
each variable \(u_{ij}\) in \(P(U)\) with \(\frac{\partial}{\partial u_{ij}}\), and
\(\overline{Q}(U)\) is the polynomial obtained by taking the complex conjugates of
the coefficients of \(Q(U)\).
Then we have
\[
  \langle \rho_{n,\bsk}(g) P(U), Q(U) \rangle
  = \langle P(U), \rho_{n,\bsk}(g^*) Q(U) \rangle
\]
for any \(g \in \GL_n(\bbC)\) and any \(P(U), Q(U) \in \bbC[U]_{\bsk}\)
(cf. \cite{atobe2023harder}).

We fix a basis \(\{v_1,\ldots, v_t\}\) of the free module \(\bbZ[U]_{\bsk}\);
for example, one may take the basis associated with semi-standard Young tableaux
(cf. \cite[\S 4.5]{green2007polynomial}).
Let \(K\) be a number field and let \(\mathcal{O}\) denote the ring of integers of \(K\).
For a prime ideal \(\frakp\) of \(\mathcal{O}\) and an element
\(a = \sum_{i=1}^t a_i v_i \in K[U]_{\bsk}\),
we define the valuation of \(a\) with respect to \(\frakp\) by
\[
  v_{\frakp}(a) = \min_{1 \le i \le t} v_{\frakp}(a_i).
\]
We say that \(\frakp\) divides \(a\) if \(v_{\frakp}(a) > 0\), and write
\(\frakp | a\).
This valuation is independent of the choice of the basis of \(\bbZ[U]_{\bsk}\).

For a family \((\bsk,\bsl)=(\kv,\lv)_{\va}\) of pairs of dominant integral weights
such that \(\ell(\kv)\leq n\) and \(\ell(\lv)\leq n\) for any \(v\in \bfa\),
we define the representation
\(\rho_{n,(\bsk,\bsl)}=\bigboxtimes_{\va}(\rho_{n, \kv}\boxtimes\rho_{n, \lv})\) of \(\tilk_{n,\infty}^\bbC\).
We put
\[M_{(\bsk,\bsl)}(\tilgam_n,\omega)=M_{\rho_{n,(\bsk,\bsl)}}(\tilgam_n,\omega)
  \quad\text{and}\quad
  S_{(\bsk,\bsl)}(\tilgam_n,\omega)=S_{\rho_{n,(\bsk,\bsl)}}(\tilgam_n,\omega).\]
When \(\bsk=(\kappa_v,\ldots,\kappa_v)_{\va}\) and \(\bsl=(0,\ldots,0)_{\va}\)
for a family \(\kappa=(\kappa_v)_{\va}\) of non-negative integers,
we also put
\[ {\det}^\kappa=\rho_{n,(\bsk,\bsl)},
  \quad M_\kappa(\tilgam_n,\omega)=M_{(\bsk,\bsl)}(\tilgam_n,\omega)
  \quad\text{and}\quad
  S_\kappa(\tilgam_n,\omega)=S_{(\bsk,\bsl)}(\tilgam_n,\omega).\]

\subsection{As functions on unitary groups over the adeles}\label{subsec:adeles}
Let \(\frakK_{n,\infty}\) be the stabilizer of \(\bfi_n \in \hus{n}^\bfa\) in \(G_{n,\infty}\).
Then \(\frakK_{n,\infty}\) is a maximal compact subgroup of \(G_{n,\infty}\) (it is also a maximal compact subgroup of \(\tilg^+_{n,\infty}\))
and is isomorphic to \(\prod_{\va}\rmu(n)\times\rmu(n)\), which is given by
\[
  \begin{array}{rccc}
     & \prod_{\va}\rmu(n)\times\rmu(n) & \rightarrow & \frakK_{n,\infty} \\
     & (k_{1,v},k_{2,v})_{\va}         & \mapsto     & \left(\frakc
    \begin{pmatrix}k_{2,v} & 0 \\0&{}^tk_{1,v}^{-1}
    \end{pmatrix}\frakc^{-1}\right)_{\va},
  \end{array}\]
where \(\frakc=\dfrac{1}{\sqrt{2}}
\begin{pmatrix}1 & \sqrt{-1} \\ \sqrt{-1} & 1\\
\end{pmatrix}\in \M_{2n}(\bbC)\).

The stabilizer of \(\bfi_n\) in \(\tilg^+_{n,\infty}\) is \(\tilk^+_{n,\infty}=\prod_{\va}(\bbR_{>0}\cdot \tilk_{n,\infty})\).
Note that \(\tilk^+_{n,\infty}\) is clearly not a maximal compact subgroup of \(\tilg^+_{n,\infty}\).

For an infinite place \(\va\), we put \(\frakg_{n,v}=\mathrm{Lie}(\tilg_{n,v})\),
\(\frakk_{n,v}=\mathrm{Lie}(\tilk_{n,v})\) and let \(\frakg^\bbC_{n,v}\) and
\(\frakk^\bbC_{n,v}\) be the complexification of \(\frakg_{n,v}\) and    \(\frakk_{n,v}\), respectively.

We have the Cartan decomposition \(\frakg_{n,v}=\frakk_{n,v}\oplus\frakp_{n,v}\).

We put
\begin{align*}
  \kappa_{v,ij}  & =\frakc
  \begin{pmatrix}0 & 0 \\  0& -e_{ji}\\
  \end{pmatrix}\frakc^{-1},\quad
  \kappa'_{v,ij}=\frakc
  \begin{pmatrix}e_{ij} & 0 \\ 0 &0 \\
  \end{pmatrix}\frakc^{-1}, \\
  \pi^{+}_{v,ij} & =\frakc
  \begin{pmatrix}0 & e_{ij} \\ 0 & 0\\
  \end{pmatrix}\frakc^{-1},\quad \mathrm{and} \quad
  \pi^{-}_{v,ij}=\frakc
  \begin{pmatrix}0 & 0 \\ e_{ij} & 0\\
  \end{pmatrix}\frakc^{-1},
\end{align*}
where
\(e_{ij} \in \M_{n,n}(\bbC)\) is the matrix whose only non-zero entry is \(1\) in the \((i,j)\)-entry.
\(\{\kappa_{v,ij}\}\) is a basis of \(\frakk^\bbC_{n,v}\).
Let \(\frakp^+_{n,v}\) (resp. \(\frakp^-_{n,v}\)) be the \(\bbC\)-span of
\(\{\pi^{+}_{v,ij}\}\) (resp. \(\{\pi^{-}_{v,ij}\}\)) in \(\frakg^\bbC_{n,v}\).
Then, we have \(\frakp^\bbC_{n,v}=\frakp^+_{n,v}\oplus\frakp^-_{n,v}\oplus\bbC\cdot I_{2n}\).
We set \(\frakg_{n}=\prod_{\va}\frakg_{n,v}\), \(\frakk^\bbC_{n}=\prod_{\va}\frakk^\bbC_{n,v}\), etc.
\begin{definition}
  Let \(\rho\) be a representation of \(\tilk_{n,\infty}\) on a finite-dimensional complex vector space \(V\),
  and let \(\omega\) be a unitary character of \(\bbR^\bfa\).

  A \emph{Hermitian automorphic form} on \(\tilg^+_{n}(\bbA)\) of weight \((\rho,V)\) and character \(\omega\)
  is a \(V\)-valued smooth function \(f : \tilg^+_{n}(\bbA) \;\longrightarrow\; V\)
  satisfying the following conditions:
  \begin{itemize}
    \item  \(f(u gk) = \omega(\nu(k_\infty))^{-1}\rho(\nu(k_\infty)^{-1/2}k_\infty)^{-1}f(g)\) for all \(u \in \tilg^+_{n}\), \(g \in \tilg^+_{n}(\bbA)\), \(k \in \tilk^+_{n,\bbA}\), where \(k_\infty\) is the infinite part of \(k\).
    \item  \(f\) is of moderate growth.
    \item  \(f\) is \(Z(\frakg)\)-finite,
          where \(Z(\frakg_n)\) denotes the center of the universal enveloping algebra of
          \(\frakg_n\)
  \end{itemize}
  We denote by \(\widetilde{\calA}_n(\rho,\omega)\) the complex vector space of Hermitian automorphic forms on \(\tilg^+_{n}(\bbA)\) of weight \(\rho\) and character \(\omega\).
\end{definition}

\begin{definition}
  A Hermitian automorphic form \(f\in\widetilde{\calA}_n(\rho,\omega)\) is called a cusp form if
  \[\int_{N(F)\backslash N(\bbA)}f(ng)dn=0\]
  for any \(g \in \tilg_{n}(\bbA)\) and the unipotent radical \(N\) of any proper parabolic subgroup of \(\tilg_n\).
  We denote by \(\widetilde{\calS}_{n}(\rho,\omega)\) the complex vector space of cusp
  forms on \(\tilg^+_{n}(\bbA)\) of weight \((\rho,V)\) and character \(\omega\).
\end{definition}
\begin{remark}
  When no character is specified, we simply write
  \[
    \widetilde{\calA}_n(\rho) = \bigoplus_{\omega}\, \widetilde{\calA}_n(\rho,\omega),\quad \text{and}\quad
    \widetilde{\calS}_{n}(\rho) = \bigoplus_{\omega}\, \widetilde{\calS}_{n}(\rho,\omega).
  \]
\end{remark}

We also define Hermitian automorphic forms for the unitary group \(G_n\) in a similar way.
Let \(\calA_n(\rho)\) (resp. \(\calS_{n}(\rho)\)) denote the space of Hermitian automorphic forms (resp.~cusp forms) on \(G_{n}(\bbA)\) of weight \((\rho, V)\).

For each finite place \(v\in \bfh\), we take the Haar measure \(dg_v\) on \(G_{n,v}\),
normalized so that the maximal compact subgroup \(\frakK_{n,v}\) has volume \(1\).

For each archimedean place \(v\in \bfa\), we choose the Haar measure \(dg_v\) on
\(G_{n,v}\) so that the volume of \(\frakK_{n,v}\) is \(1\), and the
induced measure on the symmetric space
\(\hus{n}\cong
G_{n,v}/\frakK_{n,v}\)
is given by
\((\det Y_v)^{-2n}\, dZ_v\) on \(\hus{n}\).

By taking the restricted product, we obtain a Haar measure
\(dg = \prod_v dg_v\)
on \(G_{n}(\bbA)\).

Using this choice, we define the \emph{Petersson inner product} on the space of
automorphic forms \(\calA_n(\rho)\) by
\[
  (f,h)
  \;=\;
  \int_{G_{n}\backslash G_{n}(\bbA)}
  \langle f(g),h(g)\rangle \, dg,
\]
for \(f,h \in \calA_n(\rho)\), where \(dg\) denotes the Haar measure on
\(G_{n}\backslash G_{n}(\bbA)\) induced from \(dg\) on \(G_{n}(\bbA)\).
If either \(f\) or \(h\) belongs to \(\calS_{n}(\rho)\), then the integral converges absolutely.

\vskip.5\baselineskip

Let \(U_1\) be the algebraic group defined by
\[
  U_1(R) = \{\, x \in (K\otimes_F R)^\times \mid x\overline{x}=1 \,\}
\]
for any \(F\)-algebra \(R\), and let \(\mathbb{G}_m\) denote the multiplicative group.
By the strong approximation theorem for the special unitary group \(\mathrm{SU}(n,n)\) and the short exact sequence
\[
  1 \;\longrightarrow\; \mathbb{G}_m \;\longrightarrow\; \mathrm{Res}_{K/F}(\mathbb{G}_m) \;\longrightarrow\; U_1 \;\longrightarrow\; 1,
\]
we obtain an isomorphism
\begin{align}\label{eq:representative}
        & \tilg^+_{n}(F)\backslash \tilg^+_{n}(\bbA) /
  \tilk_{n,\bbA}\tilg^+_{n,\infty}                                                                                          \\
  \cong & \;
  \bigl(F^\times\backslash\bbA^\times / \prod_{v\in\bfh}\calO_v^\times\cdot\bbA^\times_\infty\bigr)
  \ltimes  \bigl(U_1(F)\backslash U_1(\bbA) / \prod_{v\in \bfh}U_1(\calO_{v}^\times)\cdot U_1(\bbA_\infty) \bigr) \nonumber \\
  \cong & \;
  \bigl(F^\times\backslash\bbA^\times / \prod_{v\in\bfh}\calO_v^\times\cdot\bbA^\times_\infty\bigr)
  \ltimes \bigl(K^\times\backslash \bbA_{K}^\times / \bbA^\times \prod_{v\in\bfh}\calO_{K_v}^\times\bbA^\times_{K,\infty}\bigr),
\end{align}

We now describe an explicit set of representatives for the relevant double coset space.
Fix complete sets of representatives
\[
  N_0:=\{\nu_1=1, \ldots, \nu_{h_1}\} \subset \bbA_0^\times
  \quad\text{for } F^\times\backslash\bbA^\times / \prod_{v\in\bfh}\calO_v^\times\cdot\bbA^\times_\infty,
\]
and
\[
  Z_0:=\{z_1=1, \ldots, z_{h_2}\} \subset \bbA_{K,0}^\times
  \quad\text{for } K^\times\backslash \bbA_{K}^\times / \bbA^\times \prod_{v\in\bfh}\calO_{K_v}^\times\bbA^\times_{K,\infty}.
\]

By the decomposition obtained above, each double coset in
\begin{equation}\label{eq:doublecoset}
  \tilg^+_{n}(F)\backslash \tilg^+_{n}(\bbA) /
  \tilk_{n,\bbA}\tilg^+_{n,\infty}
\end{equation}
is uniquely determined by a choice of parameters \(\nu_i\) and \(z_j\).
Consequently, a complete set of double coset representatives may be indexed by pairs \((i,j)\).
More precisely, we may choose a set of representatives \(\gamma_{ij}\) in~\eqref{eq:doublecoset} as

\[
  \frakG_0=\frakG^{(n)}_0:=\left\{\gamma_{ij}
  = \left( \bfs(\nu_{i}) \, \bft(z_{j}) \right)
  \middle|\,
  \nu_i \in N_0, \, z_j \in Z_0
  \right\}
  \subset \tilg^+_{n,0},
\]
where
\[
  \bfs(\nu)
  =
  \begin{pmatrix}
    \nu I_n & 0   \\
    0       & I_n
  \end{pmatrix},
  \qquad
  \bft(z)
  =
  \begin{pmatrix}
    I_{n-1} & 0 & 0       & 0                 \\
    0       & z & 0       & 0                 \\
    0       & 0 & I_{n-1} & 0                 \\
    0       & 0 & 0       & \overline{z}^{-1}
  \end{pmatrix}.
\]
We fix an ordering of the set of pairs \((i,j)\) and, with respect to this ordering, occasionally denote \(\gamma_{ij}\) by \(\gamma_s\).

In particular, the subset \(\{\gamma_{1j}\}_{1\le j \le h_2}\) forms a complete set of representatives for the double coset
\[
  G_{n}(F)\backslash G_{n}(\bbA) /
  \frakK_{n,\bbA} G_{n,\infty}.
\]

For each \(i,j\), set
\[
  \tilgam_{ij}
  = \tilg^+_n(F)\cap
  \bigl(\gamma_{ij} \tilk_{n,\bbA}\gamma_{ij}^{-1}\cdot \tilg^+_{n,\infty}\bigr).
\]

Define
\[
  M_\rho(\tilgam_{s},\omega)
  = \Bigl\{\, f : \hus{n}^\bfa \to V \;\Bigm|\;
  f|_{(\rho,\omega)} \gamma = f
  \ \text{for all } \gamma \in \tilgam_{s} \Bigr\},
\]
as before.

Given \((f_1, \ldots, f_h) \in \bigoplus_{i=1}^h M_\rho(\tilgam_{i},\omega)\),
we define a function
\[
  (f_1, \ldots, f_h)^\sharp : \tilg^+_n(\bbA) \longrightarrow V
\]
by
\begin{equation}\label{eq:fsharp}
  (f_1,\ldots,f_h)^\sharp(g)
  = f_s \mid_{\rho,\omega} g_\infty(\bfi_n)
  = \omega(\nu(g_\infty))^{-1} \, \rho(M(g_\infty))^{-1} \,
  f_s\bigl(g_\infty \langle \bfi_n \rangle\bigr),
\end{equation}
for
\[
  g = u \gamma_s k g_\infty, \quad
  u \in \tilg^+_n(F), \quad
  \gamma_s \in \frakG_0, \quad
  k \in \tilk_{n,0}, \quad
  g_\infty \in \tilg^+_{n,\infty}.
\]

Then \((f_1,\ldots,f_h)^\sharp\) belongs to \(\widetilde{\calA}_n(\rho,\omega)\).
We denote the image of \(\bigoplus_{i=1}^h M_\rho(\tilgam_{i},\omega)\)
(resp.\ \(\bigoplus_{i=1}^h S_\rho(\tilgam_{i},\omega)\))
under the map \(\sharp\) in~\eqref{eq:fsharp}
by \(\widetilde{\calA}_n^{\sharp}(\rho,\omega)\) (resp.\ \(\widetilde{\calS}_{n}^{\sharp}(\rho,\omega)\)).
We also define \(\calA_n^{ \sharp}(\rho)\) and \(\calS_{n}^{\sharp}(\rho)\) for \(G_n\) in a similar way.

Let \(\Her_n(K)\) be the set of all Hermitian matrices in \(\M_n(K)\),
and \(\Her_n(\bbA_K)\) be the set of all Hermitian matrices in \(\M_n(\bbA_K)\).
We fix a non-trivial additive character
\(\bfe_\bbA(x)=\prod_{v\in\bfh}\psi_v(x_v)\times\prod_{\va}\bfe(x_v)\) of \(\bbA/F\),
such that \(\psi_v\) is trivial on \(\calO_{F_v}\) for any finite place \(v\in\bfh\).
For \(S\in \Her_n(K)\), we define the \emph{adelic \(S\)-th Fourier coefficient} of
\(f\in \widetilde{\calA}_n(\rho,\omega)\) by
\begin{equation}\label{eq:fourier}
  W_f(g,S)=\int_{\Her_n(K)\backslash\Her_n(\bbA\otimes_F K)} f(\bfn(X)g)\overline{\bfe_\bbA(\tr(SX))}\,dX ,
\end{equation}
for \(g\in \tilg^+_{n}(\bbA)\), where \(\bfn(X)=
\begin{pmatrix}I_n & X \\ 0 & I_n
\end{pmatrix}\).
We take \(dX\) to be the self-dual Haar measure with respect to the pairing
\((S,X)\mapsto \bfe_\bbA(\tr(SX))\).
With this normalization, Fourier inversion gives the following lemma.

\begin{lemma}\label{lem:Af}
  \begin{enumerate}
    \item We have
          \[
            f(g)=\sum_{S\in\Her_n(K)}W_f(g,S).
          \]
    \item For any \(X\in \Her_n(\bbA\otimes_F K)\) and \(k\in \tilk_{n,0}\),
          we have
          \[
            W_f(\bfn(X)gk,S)=\bfe_\bbA(\tr(SX))W_f(g,S).
          \]
    \item For any \(\alpha\in \GL_n(K)\) and \(\nu\in F_{>0}^\times\),
          we have
          \[
            W_f(\bfs(\nu)\bfm(\alpha)g,S)=W_f(g,\nu\,S[\alpha]),
          \]
          where
          \[
            \bfs(\nu)=
            \begin{pmatrix} \nu I_n & 0 \\ 0 & I_n\end{pmatrix},
            \qquad
            \bfm(\alpha)=
            \begin{pmatrix}\alpha & 0 \\ 0 & \adj{\alpha}^{-1}\end{pmatrix},
          \]
          and \(S[\alpha]=\adj{\alpha} S \alpha\).
  \end{enumerate}
\end{lemma}

We put
\[
  \eta(g_\infty,S)
  =\bfe(\tr(Sg_\infty\left<\bfi_n\right>))
  \omega(\nu(g_\infty))^{-1}\rho(M(g_\infty))^{-1}
\]
for \(S\in \Her_n(\bbC)^\bfa\) and \(g_\infty\in \tilg^+_{n,\infty}\).
When the weight needs to be indicated explicitly, we write \(\eta_{\rho,\omega}(g_\infty,S)\), or simply \(\eta_\rho(g_\infty,S)\) when \(\omega\) is trivial.
\begin{lemma}\label{lem:eta}
  For \(g_\infty\in \tilg^+_{n,\infty}\), \(X_\infty\in \Her_n(\bbC)^\bfa\),
  \(\nu_\infty\in (\bbR_{>0})^\bfa\), and \(A_\infty\in \GL_n(\bbC)^\bfa\),
  we have
  \begin{align*}
    \eta(\bfn(X_\infty)g_\infty,S)
     & =\bfe(\tr(SX_\infty))\eta(g_\infty,S),                        \\
    \eta(\bfs(\nu_\infty)\bfm(A_\infty)g_\infty,S)
     & =\omega(\nu_\infty)^{-1}\eta(g_\infty,\nu_\infty S[A_\infty])
    \rho(\nu_\infty^{1/2}\adj{A_\infty},\nu_\infty^{1/2}\,{}^t\!A_\infty).
  \end{align*}
\end{lemma}

For \(g=g_0g_\infty\in\tilg_n^+(\bbA)\), with \(g_0\in\tilg^+_{n,0}\) and \(g_\infty\in\tilg^+_{n,\infty}\),
we define the \emph{normalized Fourier coefficient}
\begin{equation}\label{eq:arithmetic_fourier}
  A_f(g,S):=\eta(g_\infty,S)^{-1}W_f(g,S).
\end{equation}
Thus the archimedean factor is removed from \(A_f(g,S)\).
In particular,
\begin{equation}\label{eq:fourier_normalized_expansion}
  f(g)=\sum_{S\in\Her_n(K)}\eta(g_\infty,S)A_f(g,S).
\end{equation}
For a finite adelic point \(g_0\), we use the abbreviation
\[
  A_f(g_0,S)=\eta(1,S)^{-1}W_f(g_0,S).
\]

We note that for \(g_0 \in \tilg^+_{n,0}\), the Iwasawa decomposition gives
\[
  g_0
  = \bfn(X_0)\,\bfs(\nu_0)\,\bfm(\alpha_0)\,\bft(x_0)\,\gamma_{ij}\,k,
\]
for some \(X_0 \in \Her_n(\bbA_{0}\otimes_F K)\), \(\nu \in F_{>0}^\times\),
\(\alpha \in \GL_n(K)\), \(x_0 \in \bbA\),
\(\gamma_{ij} \in \frakG_0\), and
\(k \in \tilk_{n,0}\).
Here \(\nu_0\) and \(\nu_\infty\) denote the finite and infinite components of
\(\nu \in F \subset \bbA\), respectively;
the same convention applies to \(\alpha\).

The classical Fourier coefficients \(a(f_{ij},T)\) and the normalized Fourier coefficients \(A_f(g,S)\) are related as follows.

\begin{proposition}\label{prop:Fourier}
  Let \(g = g_0 g_\infty \in \tilg^+_{n}(\bbA)\) with
  \(g_0 = \bfn(X_0)\,\bfs(\nu_0)\,\bfm(\alpha_0)\,\gamma_{ij}\,k\)
  as above, so that \(x_0=1\),
  and let \(g_\infty \in \tilg^+_{n,\infty}\).
  Then, for
  \(f = (f_{11},\ldots,f_{h_1h_2})^\sharp \in \widetilde{\calA}_n^{\sharp}(\rho,\omega)\),
  we have
  \[
    A_f(g,S)
    = \bfe_\bbA(\tr(SX_0))\,
    \omega(\nu_\infty)\,
    \rho(\nu_\infty^{1/2}\,\adj{\alpha_\infty},\nu_\infty^{1/2}\,^t\!\alpha_\infty)^{-1}\,
    a(f_{ij},\nu S[\alpha]).
  \]
  In particular,
  \[
    A_f(\gamma_{ij},S)=a(f_{ij},S).
  \]
\end{proposition}

\begin{proof}
  Put
  \[
    f_{ij,\infty}(g_\infty)
    = f(\gamma_{ij}g_\infty)
    = f_{ij}|_{\rho,\omega} g_\infty(\bfi_n).
  \]
  Then
  \[
    f_{ij,\infty}(\bfn(X_\infty)g_\infty)
    = \sum_{T\in \Her_n(K)} \eta(g_\infty,T)\,a(f_{ij},T)\,
    \bfe\!\left(\sum_{v\in\bfa}\tr(T_{\sigma_v}X_v)\right).
  \]
  On the other hand,
  Lemma~\ref{lem:Af} gives
  \begin{align*}
    f_{ij,\infty}(\bfn(X_\infty)g_\infty)
     & = \sum_{S\in\Her_n(K)} W_f(\gamma_{ij}\bfn(X_\infty)g_\infty,S) \\
     & = \sum_{S\in\Her_n(K)} W_f(\gamma_{ij}g_\infty,S)\,
    \bfe\!\left(\sum_{v\in\bfa}\tr(S_{\sigma_v}X_v)\right).
  \end{align*}
  Comparing the two Fourier expansions yields
  \[
    W_f(\gamma_{ij}g_\infty,S)
    = \eta(g_\infty,S)\,a(f_{ij},S).
  \]
  Hence
  \[
    A_f(\gamma_{ij}g_\infty,S)=a(f_{ij},S).
  \]

  Now let \(g=g_0g_\infty\) be as in the statement.
  By Lemmas~\ref{lem:Af} and~\ref{lem:eta},
  \begin{align*}
    A_f(g,S)
     & = \bfe_\bbA(\tr(SX_0))\,
    \eta(g_\infty,S)^{-1}\,
    W_f(\bfs(\nu_0)\bfm(\alpha_0)\gamma_{ij}g_\infty,S)                                 \\
     & = \bfe_\bbA(\tr(SX_0))\,
    \eta(g_\infty,S)^{-1}\,
    W_f(\gamma_{ij}\bfs(\nu^{-1}_\infty)\bfm(\alpha^{-1}_\infty)g_\infty,\nu S[\alpha]) \\
     & = \bfe_\bbA(\tr(SX_0))\,
    \eta(g_\infty,S)^{-1}\,
    \eta(\bfs(\nu^{-1}_\infty)\bfm(\alpha^{-1}_\infty)g_\infty,\nu_\infty S_\infty[\alpha_\infty])\,
    a(f_{ij},\nu S[\alpha])                                                             \\
     & = \bfe_\bbA(\tr(SX_0))\,
    \omega(\nu_\infty)\,
    \rho(\nu_\infty^{1/2}\,\adj{\alpha_\infty},\nu_\infty^{1/2}\,^t\!\alpha_\infty)^{-1}\,
    a(f_{ij},\nu S[\alpha]).
  \end{align*}
  This proves the assertion.
\end{proof}

\section{Hecke operators on Hermitian automorphic forms}\label{sec:hecke-operators}
For a finite place \(v \in \bfh\) corresponding to a prime ideal \(\frakp\) of \(F\),
let \(\calH_{n,\frakp}\) be the convolution algebra of compactly supported \(\bbZ\)-valued functions on \(\tilg_{n,v}\) that are left and right \(\tilk_{n,v}\)-invariant and supported on \(\tilg_{n,v} \cap \M_{2n}(\calO_{K_v})\),
which is called the \emph{spherical Hecke algebra at \(\frakp\)}.
The elements of \(\calH_{n,\frakp}\) are called Hecke operators at \(\frakp\).
The spherical Hecke algebra \(\calH_{n,\frakp}\) acts from the left on the space of
automorphic forms \(\widetilde{\calA}_n(\rho_{n,(\bsk,\bsl)})\) as follows:
\[
  (T\cdot f)(g) = \int_{\tilg_{n,v}} f(gh^{-1})\,T(h)\, |\nu(h)|_{\frakp}^{-r_n}\,dh,
\]
where we put \(r_n=r_{n,(\bsk,\bsl)} = \sum_{i=1}^n (k_i + l_i)/2 - n^2\), and \(dh\) is the Haar measure on \(\tilg_{n,v}\) normalized so that the volume of \(\tilk_{n,v}\) is 1.

Note that the characteristic functions \(\mathbf{1}_{\tilk_{n,v}g\tilk_{n,v}}\),
which generate the Hecke algebra, can be naturally identified with the
corresponding double cosets \(\tilk_{n,v}g\tilk_{n,v}\).

For \((f_1,\ldots,f_h)^\sharp\), the Hecke operator
\(T=\tilk_{n,v}g\tilk_{n,v}\) with \(g\in \tilg_{n,v}\) acts as follows.
If the double coset decomposes as
\[
  \tilk_{n,v}g\tilk_{n,v}
  = \bigsqcup_{i\in I}\tilk_{n,v}g_i,
  \quad g_i \in \tilg_{n,v},
\]
then there exist \(u_{i,j}\in \tilg^+_{n}(F)\) and
\(s_{i,j}\in\{1,\ldots,h\}\) such that
\[
  \gamma_j g_i^{-1} \in u_{i,j}\gamma_{s_{i,j}}\tilk_{n,0}\tilg^+_{n,\infty}.
\]
Accordingly, \(T\) acts on \((f_1,\ldots,f_h)^\sharp\) as
\begin{align*}
    & T\cdot(f_1,\ldots,f_h)^\sharp \\
  = & \left(
  \sum_{i\in I}|N_{F/\bbQ}(\nu(u_{i,1}))|^{-r_n}\cdot f_{s_{i,1}}|_{\rho,\omega} u_{i,1}^{-1}, \;
  \ldots,\;
  \sum_{i\in I} |N_{F/\bbQ}(\nu(u_{i,h}))|^{-r_n}\cdot f_{s_{i,h}}|_{\rho,\omega} u_{i,h}^{-1}
  \right)^\sharp.
\end{align*}

\begin{definition}
  We say that a continuous right \(\tilk_{n,v}\)-invariant function \(f\) on \(\tilg_{n,v}\) (or on \(\tilg_{n}(\bbA)\)) is a \(\frakp\)-Hecke eigenfunction
  if \(f\) is an eigenfunction under the action of \(\calH_{n,\frakp}\).
\end{definition}

\begin{definition}
  We say that a Hermitian automorphic form \(f\in\widetilde{\calA}_n(\rho)\) (resp.
  \(f\in\widetilde{\calS}_{n}(\rho)\)) is a \emph{Hecke eigenform} (resp.
  a \emph{Hecke cusp form})
  if \(f\) is a \(\frakp\)-Hecke eigenfunction for any \(\frakp\).
\end{definition}

The structure of the Hecke algebra has been investigated in detail by Raum~\cite{raum2011hecke}.

\begin{lemma}[{\cite[Lemma~3.1]{raum2011hecke}}]\label{lem:hecke}
  Let \(v\) be a finite place of \(F\) corresponding to a prime ideal \(\frakp\) of \(F\).
  Then the spherical Hecke algebra \(\calH_{n,\frakp}=\mathcal H(\tilg_{n,v}\cap \M_{2n}(\calO_{K_v}),\, \tilk_{n,v})\) is described as follows:
  \begin{enumerate}
    \item If \(\frakp\) is inert in \(K\), then \(\calH_{n,\frakp}\) is generated by \(\{\,T(\varpi),\ T_1(\varpi^2),\ \ldots,\ T_n(\varpi^2)\,\}\),
          where
          \[
            T(\varpi)
            = \tilk_{n,v}
            \; \diag(\varpi I_n,  I_n)\;
            \tilk_{n,v},
          \]
          and
          \[
            T_i(\varpi^2)
            = \tilk_{n,v}\; \diag(\varpi^2 I_{n-i},\varpi I_{i}, I_{n-i}, \varpi I_{i})\;
            \tilk_{n,v}
          \]
          for \(1 \le i \le n\).
          Here \(\varpi\) denotes a uniformizer of \(K_\frakp\), and we assume \(\varpi \in F\).

    \item If \(\frakp\) is ramified in \(K\) as \(\frakP^2\), then \(\calH_{n,\frakp}\) is generated by \(\{T_0(\varpi),T_1(\varpi),\ldots,T_n(\varpi)\}\),
          where
          \[\qquad T_i(\varpi)=\tilk_{n,v}
            \; \diag(\varpi\overline{\varpi} I_{n-i},\varpi I_{i}, I_{n-i}, \varpi I_{i})\;
            \tilk_{n,v},\]
          where \(\varpi\) is a uniformizer of \(K_\frakP\) for \(0\leq i\leq n\).
    \item If \(\frakp\) splits in \(K\) as \(\frakP\overline{\frakP}\), then \(\calH_{n,\frakp}\cong\calH(\mathrm{Inv}_n^{\mathrm{Res}}(\calO_v),\GL_{2n}(\calO_v))\), where
          \[\mathrm{Inv}_n^{\mathrm{Res}}(\calO_v)=\{(g,l)\in (\GL_{2n}(F_v)\cap \M_{2n}(\calO_v))\times \bbZ_{\geq0}\mid \varpi^l g^{-1}\in \M_{2n}(\calO_v)\}\]
          and \(\GL_{2n}(\calO_v)\) acts trivially on the second component of \(\mathrm{Inv}_n^{\mathrm{Res}}(\calO_v)\).
          In this case, \(\calH_{n,\frakp}\) is generated by \(\{T_0(\varpi),\ldots,T_{2n}(\varpi)\}\), where
          \[T_i(\varpi)=\GL_{2n}(\calO_v)
            \; \diag(\varpi I_i, I_{2n-i})\;
            \GL_{2n}(\calO_v)\]
          for \(0\leq i\leq 2n\).
          The second component of each \(T_i(\varpi)\) is 1.
          Here, \(\varpi\) is a uniformizer of \(F_v\).
  \end{enumerate}
\end{lemma}

To study the action of these Hecke operators on
\(f = (f_1,\ldots,f_h)^\sharp \in \widetilde{\calA}_n(\rho)\) more concretely,
we now describe the left cosets of the generators introduced in the lemma above
(cf.\ Freitag~\cite[IV, \S 3]{freitag1983siegelsche}).

For an inert or ramified place \(v\) corresponding to a prime ideal \(\frakp\) of \(F\),
let \(\mathfrak{m} \subset \tilg_{n,v} \cap \M_{2n}(\calO_{K_v})\) be a subset which is left and right \(\tilk_{n,v}\)-invariant and whose elements have the same similitude \(\nu\).
We set
\[
  \frakA (\mathfrak{m}) =
  \left\{ A \in \GL_n(K_v)\cap\M_n(\calO_{K_v}) \middle|
  \begin{pmatrix}  A & * \\ 0 & \nu\cdot A^{*-1}
  \end{pmatrix} \in \mathfrak{m} \right\}.
\]

For \(A \in \frakA(\mathfrak{m})\), we define
\[
  \frakB(A, \mathfrak{m}) =
  \left\{ B \in \M_n(\calO_{K_v}) \middle|
  \begin{pmatrix}  A & B \\ 0 & \nu\cdot A^{*-1}
  \end{pmatrix} \in \mathfrak{m} \right\}.
\]

On \(\frakB(A, \mathfrak{m})\), we introduce an equivalence relation with respect to \(D \in \M_n(\calO_{K_v})\) by
\[
  B \sim_D B'
  \quad \Longleftrightarrow \quad
  B - B' \in \Her_n(\calO_{K_v}) D.
\]

\begin{lemma}\label{lem:coset_representatives}
  Assume that \(v\) is inert or ramified in \(K\), and let
  \(\mathfrak{m} \subset \tilg_{n,v} \cap \M_{2n}(\calO_{K_v})\)
  be a subset which is left and right \(\tilk_{n,v}\)-invariant and whose elements have the same similitude \(\nu\).
  Then a complete set of left coset representatives of
  \(\tilk_{n,v} \backslash \mathfrak{m}\)
  is given by
  \[
    \left\{
    \begin{pmatrix} A & B \\ 0 & D
    \end{pmatrix}
    \middle|
    \begin{array}{l}
      A \;\text{runs through a complete set of representatives of }
      \GL_n(\calO_{K_v})\backslash \frakA (\mathfrak{m}), \\[6pt]
      D = \nu\cdot A^{*-1},                               \\[6pt]
      B \;\text{runs through a complete set of representatives of }
      \frakB(A,\mathfrak{m})/\sim_D
    \end{array}
    \right\}.
  \]
\end{lemma}

In the rest of this section, we fix a uniformizer \(\varpi\), and for each \(i=1,\ldots,h_1\), choose a representative \(\nu_{l_i}\) and an element \(\pi_i \in F^\times\) such that
\[
  \nu_i^{-1}\varpi
  \in
  \pi_i \nu_{l_i}^{-1}
  \left(\prod_{v \in \bfh} \calO_v^\times\right)
  \bbA_\infty^\times.
\]

We fix \(f = (f_{11},\ldots,f_{h_1h_2})^\sharp \in \widetilde{\calA}_n(\rho,\omega)\) for the rest of this section.
\subsection{The case where \texorpdfstring{\(v\)}{v} is inert}
First, suppose that \(v\) (and hence \(\frakp\)) is inert in \(K\).
Define
\[
  \tau^{(n)}_s(\varpi)
  =
  \GL_n(\calO_{K_v})
  \, \diag(I_{n-s}, \varpi I_s) \,
  \GL_n(\calO_{K_v}),
\]
and
\[
  \tau^{(n)}_{s,t}(\varpi)
  =
  \GL_n(\calO_{K_v})
  \, \diag( I_{n-s-t}, \varpi I_s, \varpi^2 I_t) \,
  \GL_n(\calO_{K_v}).
\]
We put \(\calO_{K_v}^{(0)}=\{0\}\), and let \(\calO_{K_v}^{(l)}\) be a complete set of representatives
for the residue classes of \(\calO_{K_v}/\varpi^{l}\calO_{K_v}\) for any positive integer \(l\).
The following lemma follows by the same argument as in Freitag~\cite[IV, 2.7]{freitag1983siegelsche}.
Let \(q_v\) denote the cardinality of the residue field of \(F_v\).
\begin{lemma}\label{lem:decomposition_A}
  We have
  \[
    \frakA\bigl(T(\varpi)\bigr)
    = \bigsqcup_{0 \le s \le n}
    \tau^{(n)}_s(\varpi),
    \qquad
    \frakA\bigl(T_i(\varpi^2)\bigr)
    = \bigsqcup_{\substack{s+t \le n \\ s \ge i}}
    \tau^{(n)}_{s,t}(\varpi).
  \]
  Moreover, a complete set of representatives of
  \(\GL_n(\calO_{K_v}) \backslash \tau^{(n)}_s(\varpi)\)
  is given by
  \[
    \left\{
    A =
    \begin{pmatrix}
      \varpi^{k_1} &        & a_{uv}       \\
                   & \ddots &              \\
                   &        & \varpi^{k_n}
    \end{pmatrix}
    \;\middle|\;
    \begin{array}{l}
      k_1+\cdots+k_n = s, \quad k_j \in \{0,1\},           \\[4pt]
      a_{uv}\in \calO_{K_v}^{(k_v)}, \; 1 \le u < v \le n, \\[4pt]
      a_{uv}=0 \text{ if } k_u = k_v = 1
    \end{array}
    \right\}.
  \]
\end{lemma}
\begin{remark}
  If we fix \(k_1,\ldots,k_n \in \{0,1\}\),
  then the number of matrices \(A\) satisfying the above conditions is
  \(q_v^{-s(s+1)} \;\prod_{j=1}^n q_v^{2j k_j}\).
\end{remark}

Using Lemma~\ref{lem:coset_representatives}, we obtain the following counting result.

\begin{lemma}\label{lem:counting}
  \begin{enumerate}
    \item A complete set of representatives of
          \(\frakB(\diag(I_{n-s}, \varpi I_s), T(\varpi)) / \!\sim_D\) is given by
          \[
            \left\{
            B =
            \begin{pmatrix}
              B_0 & 0 \\
              0   & 0
            \end{pmatrix}
            \,\middle|\,
            B_0 \in \Her_{n-s}(\calO_{K_v}) / \varpi \Her_{n-s}(\calO_{K_v})
            \right\}.
          \]
          In particular, for any \(A \in \tau^{(n)}_{s}(\varpi)\), we have
          \[
            \#\bigl(\frakB(A, T(\varpi)) / \!\sim_D\bigr)
            = q_v^{(n-s)^2}.
          \]

    \item A complete set of representatives of
          \(\frakB(\diag(I_{n-s-t}, \varpi I_s, \varpi^2 I_t), T_i(\varpi^2)) / \!\sim_D\)
          is given by
          \[
            \left\{
            B =
            \begin{pmatrix}
              B_{1}             & B_{2} & 0 \\[6pt]
              \varpi\adj{B_{2}} &
              \begin{matrix}
                B_3 \  & 0 \\[6pt]
                0   \  & 0
              \end{matrix}
                                & 0         \\[6pt]
              0                 & 0     & 0
            \end{pmatrix}
            \,\middle|\,
            \begin{array}{l}
              B_{1} \in \Her_{n-s-t}(\calO_{K_v}) / \varpi^2 \Her_{n-s-t}(\calO_{K_v}), \\
              B_{2} \in \M_{n-s-t,s}(\calO_{K_v}) / \varpi \M_{n-s-t,s}(\calO_{K_v}),   \\
              B_{3} \in \Her_{s-i}(\calO_{K_v}) / \varpi\Her_{s-i}(\calO_{K_v}),        \\
              \det(B_3) \not\in (\varpi)
            \end{array}
            \right\}.
          \]

  \end{enumerate}
\end{lemma}
By combining the above lemmas, we obtain the following double coset decomposition:
\[
  T(\varpi)
  =
  \bigsqcup_{\substack{0 \le s \le n \\[2pt]
      A \in \GL_n(\calO_{K_v}) \backslash \tau^{(n)}_{s}(\varpi) \\[2pt]
      B \in \frakB(A,T(\varpi))/\!\sim_D}}
  \tilk_{n,v}
  \begin{pmatrix}
    A & B \\ 0 & \varpi (A^*)^{-1}
  \end{pmatrix}.
\]
Hence, for any \(f \in \widetilde{\calA}_n(\rho_{n,(\bsk,\bsl)})\) we have
\[
  (T \cdot f)(g)
  = q_v^{r_n}\,
  \sum_{s,A,B}
  f\!\left(
  g
  \begin{pmatrix}
    A & B \\ 0 & \varpi (A^*)^{-1}
  \end{pmatrix}^{-1}
  \right).
\]

For each \(\nu_i\in N_0\), choose a representative \(\nu_{l_i}\in N_0\) and an element \(\pi_i \in F^\times\) such that
\[
  \nu_i^{-1}\varpi
  \in
  \pi_i \nu_{l_i}^{-1}
  \left(\prod_{v \in \bfh} \calO_v^\times\right)
  \bbA_\infty^\times.
\]

For each
\(A \in \GL_n(\calO_{K_v}) \backslash \tau^{(n)}_{s}(\varpi)\),
there exists \(\alpha_A \in \GL_n(K)\) such that
\[
  \varpi\, \bft_0(z_j)\, A^{-1}
  \in
  \alpha_A\,
  \bft_0\!\bigl(z_j(\nu_i\nu_{l_i}^{-1})^{\,n-s}\bigr)\,
  \Bigl(\prod_{v \in \bfh} \GL_n(\calO_{K_v})\Bigr)\,
  \GL_n(\bbA_{K,\infty})
\]
and \(\det\alpha_A=\pi_i^{n-s}\),
where we set
\[
  \bft_0(z)
  =
  \begin{pmatrix}
    I_{n-1} & 0 \\[2pt]
    0       & z
  \end{pmatrix}
  \quad (z \in K),
\]
so that \(\bft(z) = \bfm(\bft_0(z))\).

Under this choice, we obtain
\begin{align*}
  \gamma_{ij}
  \begin{pmatrix}
    A & B \\ 0 & \varpi {A^*}^{-1}
  \end{pmatrix}^{-1}
   & =
  \gamma_{ij}\bfn(-A^{-1}B)\bfs(\varpi^{-1})\bfm(\varpi A^{-1}) \\[3pt]
   & =
  \bfn\!\left(-\nu_i(A^{-1}B)[\bft_0(\overline{z_j})]\right)
  \bfs(\nu_i\varpi^{-1})
  \bft(z_j)\bfm(\varpi A^{-1})                                  \\[3pt]
   & \in
  \bfn\!\left(-\nu_i(A^{-1}B)[\bft_0(\overline{z_j})]\right)
  \bfs(\pi_{i,0}^{-1})
  \bfm(\alpha_{A,0})
  \gamma_{l_ij}
  \bft\!\bigl((\nu_i\nu_{l_i}^{-1})^{n-s}\bigr)
  \tilk_{n,0},
\end{align*}
where \(\pi_{i,0}\) and \(\alpha_{A,0}\) denote the finite components of \(\pi_i\) and \(\alpha_A\), respectively.

Assume that
\[
  \bft\!\bigl((\nu_i\nu_{l_i}^{-1})^{n-s}\bigr)
  \in
  x_{is}\,\tilk_{n,0}\,h_{is},
\]
with \(x_{is} \in \tilg_n^+\) and \(h_{is} \in \tilg_{n,\infty}^+\).
Then, if we put \(q_{\pi_i}=N_{F/\bbQ}(\pi_i)\), we have
\begin{align*}
  A_{T(\varpi)\cdot f}(\gamma_{ij},S)
   & =q_{\pi_i}^{r_n}
  \sum_{s=0}^n
  \sum_{A \in \GL_n(\calO_{K_v}) \backslash \tau^{(n)}_{s}(\varpi)}
  \left(
  \sum_B
  \bfe_\bbA\!\left(
    -\tr\!\left(
      S\nu_i(A^{-1}B)[\bft_0(\overline{z_j})]
      \right)
    \right)
  \right)                  \\
   & \quad \cdot
  \omega(\pi_{i,\infty}^{-1})
  \rho(\pi_{i,\infty}^{1/2}\alpha^{*-1}_{A,\infty},\pi_{i,\infty}^{1/2}\,^t\!\alpha^{-1}_{A,\infty})
  a\!\left(
  f_{l_ij}|_{\rho,\omega}h_{is},
  \pi_i^{-1}S[\alpha_A]
  \right)                  \\[4pt]
   & =q_{\pi_i}^{2r_n+n^2}
  \sum_{s=0}^n
  \sum_{A \in \GL_n(\calO_{K_v}) \backslash \tau^{(n)}_{s}(\varpi)}
  \left(
  \sum_B
  \bfe_\bbA\!\left(
  -\tr\!\left(
  \nu_i S[\bft_0(z_j)](A^{-1}B)
  \right)
  \right)
  \right)                  \\
   & \quad \cdot
  \omega(\pi_{i,\infty}^{-1})
  \rho(\alpha^{*-1}_{A,\infty},\,^t\!\alpha^{-1}_{A,\infty})
  a\!\left(
  f_{l_ij}|_{\rho,\omega}h_{is},
  \pi_i^{-1}S[\alpha_A]
  \right)
\end{align*}
by Proposition~\ref{prop:Fourier}.
Define
\[
  Q_s(S)
  =
  \left\{
  A \in \GL_n(\calO_{K_v}) \backslash \tau^{(n)}_{s}(\varpi)
  \;\middle|\;
  \sum_{B \in \frakB(A,T(\varpi))/\!\sim_D}
  \bfe_\bbA\!\big(-\tr(SA^{-1}B)\big)
  \neq 0
  \right\}.
\]
By Lemma~\ref{lem:counting}, we obtain the following formula:
\begin{align}\label{eq:hecke_inert_tp}
  A_{T(\varpi)\cdot f}(\gamma_{ij},S)
  =q_{\pi_i}^{2r_n+n^2}\sum_{s=0}^n
  q_v^{(n-s)^2}
  \sum_{A \in Q_s(\nu_iS[\bft_0(z_j)])}
   & \omega(\pi_{i,\infty}^{-1})\,
  \rho(\alpha^{*-1}_{A,\infty},\,^t\!\alpha^{-1}_{A,\infty})\,\nonumber \\
   & \quad\cdot a\!\left(
  f_{l_ij}|_{\rho,\omega}h_{is},
  \pi_i^{-1}S[\alpha_A]
  \right).
\end{align}

The operator \(T_i(\varpi^2)\) can be treated in a similar manner.
For each \(\nu_i\in N_0\), choose a representative \(\nu_{l'_i}\in N_0\) and an element \(\pi_i^{(2)} \in F^\times\) such that
\[
  \nu_i^{-1}\varpi^2
  \in
  \pi_i^{(2)} \nu_{l'_i}^{-1}
  \left(\prod_{v \in \bfh} \calO_v^\times\right)
  \bbA_\infty^\times.
\]
Note that the subscript \((2)\) is only notational and does not indicate squaring.
For each
\(A \in \GL_n(\calO_{K_v}) \backslash \tau^{(n)}_{s,t}(\varpi)\),
there exists \(\alpha_A \in \GL_n(K)\) such that
\[
  \varpi^2 \bft_0(z_j) A^{-1}
  \in
  \alpha_A\,
  \bft_0\!\bigl(z_j(\nu_i\nu_{l'_i}^{-1})^{2n-2t-s}\bigr)
  \Bigl(\prod_{v \in \bfh}\GL_n(\calO_{K_v})\Bigr)
  \GL_n(\bbA_{K,\infty})
\]
and \(\det\alpha_A=(\pi_i^{(2)})^{2n-2t-s}\).
Assume that
\[
  \bft\!\bigl((\nu_i\nu_{l'_i}^{-1})^{2n-2t-s}\bigr)
  \in
  y_{is}\,\tilk_{n,0}h'_{is},
\]
with \(y_{is} \in \tilg_n^+\) and \(h'_{is} \in \tilg_{n,\infty}^+\).
Then, if we put \(q_{\pi_i^{(2)}}=N_{F/\bbQ}(\pi_i^{(2)})\), we obtain
\begin{align*}
  A_{T_i(\varpi^2)\cdot f}(\gamma_{ij},S)
   & = q_{\pi_i^{(2)}}^{2r_n+n^2}
  \sum_{\substack{s+t \le n                                          \\ s \ge i}}
  \sum_{A \in \GL_n(\calO_{K_v})\backslash \tau^{(n)}_{s,t}(\varpi)} \\
   & \quad \cdot
  \left(
  \sum_{B \in \frakB(A,T_i(\varpi^2))/\!\sim_D}
  \bfe_\bbA\!\left(
  -\tr\!\big(
  \nu_i S[\bft_0(z_j)](A^{-1}B)
  \big)
  \right)
  \right)                                                            \\
   & \quad \cdot
  \omega((\pi^{(2)}_{i,\infty})^{-1})
  \rho(\alpha^{*-1}_{A,\infty},\,{}^t\!\alpha^{-1}_{A,\infty})
  a\!\left(
  f_{l'_ij}|_{\rho,\omega}h'_{is},
  (\pi_i^{(2)})^{-1}S[\alpha_A]
  \right).
\end{align*}

The summation over \(B\) is generally difficult to evaluate directly,
unlike in the case of \(T(\varpi)\).
Therefore, following \cite{hafner2002explicit} and \cite{caulk2007hecke},
we introduce a new operator \(\widetilde{T_i}(\varpi^2)\)
as a linear combination of several \(T_i(\varpi^2)\).
For this purpose, set
\[
  T_0(\varpi^2)
  :=\tilk_{n,v}\,\diag(\varpi^2 I_n,I_n)\,\tilk_{n,v}.
\]
The operator \(T_0(\varpi^2)\) is introduced only for this completion.

\begin{definition}
  For \(0 \le i \le n\), we define the Hecke operator
  \[
    \widetilde{T_i}(\varpi^2)
    =
    \sum_{j=0}^i
    \begin{bmatrix}
      n-j \\ i-j
    \end{bmatrix}_{q_v^2}
    T_{j}(\varpi^2),
  \]
  where
  \[
    \begin{bmatrix}
      s \\ t
    \end{bmatrix}_{q}
    =
    \prod_{k=1}^t
    \frac{q^{s-k+1}-1}{q^k-1}
  \]
  denotes the \(q\)-binomial coefficient.
\end{definition}

By the same argument as in Proposition~5.1 of~\cite{caulk2007hecke},
we then obtain
\begin{align*}
  A_{\widetilde{T_i}(\varpi^2)\cdot f}(\gamma_{ij},S)
  = & q_{\pi_i^{(2)}}^{2r_n+n^2}
  \sum_{\substack{s+t \le n      \\ s \ge i}}
  \sum_{A \in \GL_n(\calO_{K_v}) \backslash \tau^{(n)}_{s,t}(\varpi)}
  \left(
  \sum_B
  \bfe_\bbA\!\left(
  -\tr\!\big(
  \nu_i S[\bft_0(z_j)](A^{-1}B)
  \big)
  \right)
  \right)                        \\
    & \cdot
  \omega((\pi^{(2)}_{i,\infty})^{-1})
  \rho(\alpha^{*-1}_{A,\infty},\,^t\!\alpha_{A,\infty}^{-1})
  a\!\left(
  f_{l'_ij}|_{\rho,\omega}h'_{is},
  (\pi_i^{(2)})^{-1}S[\alpha_A]
  \right),
\end{align*}
where \(B\) runs through a complete set of representatives of
\[
  \frakB(A,s,t)
  :=
  \left\{
  AH
  \;\middle|\;
  H =
  \begin{pmatrix}
    H_1             & H_2 & 0 \\[6pt]
    \varpi\adj{H_2} & H_3 & 0 \\[6pt]
    0               & 0   & 0
  \end{pmatrix},
  \quad
  \begin{array}{l}
    H_{1} \in \Her_{n-s-t}(\calO_{K_v})/
    \varpi^2\Her_{n-s-t}(\calO_{K_v}), \\[4pt]
    H_{2} \in \M_{n-s-t,s}(\calO_{K_v})/
    \varpi\M_{n-s-t,s}(\calO_{K_v}),   \\[4pt]
    H_{3} \in \Her_s(\calO_{K_v})/
    \varpi\Her_s(\calO_{K_v})
  \end{array}
  \right\}.
\]
If we define
\[
  Q_{st}(S)
  =
  \left\{
  A \in \GL_n(\calO_{K_v}) \backslash \tau^{(n)}_{s,t}(\varpi)
  \;\middle|\;
  \sum_{B \in \frakB(A,s,t)}
  \bfe_\bbA\!\big(-\tr(SA^{-1}B)\big)
  \neq 0
  \right\},
\]
then it follows that
\begin{align}\label{eq:hecke_inert_tpi}
  A_{\widetilde{T_i}(\varpi^2)\cdot f}(\gamma_{ij},S)
   & = q_{\pi_i^{(2)}}^{2r_n+n^2}
  \sum_{\substack{s+t \le n       \\ s \ge i}}
  q_v^{2(n-s-t)(n-t)+s^2}
  \sum_{A \in Q_{st}(\nu_i S[\bft_0(z_j)])}
  \omega((\pi^{(2)}_{i,\infty})^{-1})
  \nonumber                       \\[3pt]
   & \quad \cdot
  \rho(\alpha^{*-1}_{A,\infty},\,{}^t\!\alpha_{A,\infty}^{-1})
  a\!\left(
  f_{l'_ij}|_{\rho,\omega}h'_{is},
  (\pi_i^{(2)})^{-1}S[\alpha_A]
  \right),
\end{align}
since \(\frakB(A,s,t)\) contains
\(q_v^{2(n-s-t)(n-t)+s^2}\) elements.
Notice that the relations expressing \(\widetilde{T_i}(\varpi^2)\) in terms of \(T_0(\varpi^2),\ldots,T_i(\varpi^2)\) are unitriangular over \(\bbZ\).

\subsection{The case where \texorpdfstring{\(v\)}{v} is ramified}
Next, we consider the case where \(v\) (and hence \(\frakp\)) is ramified in \(K\).
This case can be treated in a way completely parallel to the inert case,
but we include full details for completeness.

Define
\[
  \tau^{(n)}_{s,t}(\varpi)
  = \GL_n(\calO_{K_v})
  \,\diag(I_{n-s-t},\, \varpi I_s,\, \varpi\overline{\varpi} I_t)\,
  \GL_n(\calO_{K_v}).
\]

Choose a representative \(\nu_{l_i}\in N_0\) and an element \(\pi_i \in F^\times\) such that
\[
  \nu_i^{-1}\varpi\overline{\varpi} \in
  \pi_i \nu_{l_i}^{-1}
  \left(\prod_{v \in \bfh} \calO_v^\times\right)
  \bbA_\infty^\times.
\]
For each \(A \in \GL_n(\calO_{K_v}) \backslash \tau^{(n)}_{s,t}(\varpi)\) and representative \(z_j\in Z_0\),
there exist a representative \(z_{m_{sj}}\in Z_0\), \(\nu_{stj}\in N_0\), and \(\alpha_A \in \GL_n(K)\) such that
\[
  \varpi\overline{\varpi}\,\bft_0(z_j)\,A^{-1}
  \in
  \alpha_A\,
  \bft_0\bigl(z_{m_{sj}}\nu_{stj}\bigr)\,
  \Bigl(\prod_{v \in \bfh}\GL_n(\calO_{K_v})\Bigr)\,
  \GL_n(\bbA_{K,\infty}).
\]
Assume that
\[
  \bft\!\bigl(\nu_{stj}\bigr)
  \in y_{stj}\,\tilk_{n,0}h'_{stj},
\]
with \(y_{stj} \in \tilg_n^+\) and \(h'_{stj} \in \tilg_{n,\infty}^+\).

\begin{definition}
  For \(0 \le i \le n\), define the Hecke operator
  \[
    \widetilde{T_i}(\varpi)
    =
    \sum_{j=0}^i
    \begin{bmatrix}
      n-j \\ i-j
    \end{bmatrix}_{q_v}
    T_j(\varpi).
  \]
\end{definition}

Now define the auxiliary set
\[
  \frakB(A,s,t)
  :=
  \left\{
  AH \;\middle|\;
  H=
  \begin{pmatrix}
    H_1                        & H_2 & 0 \\[6pt]
    \overline{\varpi}\adj{H_2} & H_3 & 0 \\[6pt]
    0                          & 0   & 0
  \end{pmatrix},
  \quad
  \begin{array}{l}
    H_1\in \Her_{n-s-t}(\calO_{K_v})/
    \varpi\overline{\varpi}\Her_{n-s-t}(\calO_{K_v}), \\[4pt]
    H_2\in \M_{n-s-t,s}(\calO_{K_v})/
    \varpi\M_{n-s-t,s}(\calO_{K_v}),                  \\[4pt]
    H_3\in \{X\in\M_s(\calO_{K_v})\mid                \\[-1pt]
    \qquad \adj X=(\overline{\varpi}/\varpi)X\}/
    \varpi\Her_s(\calO_{K_v})
  \end{array}
  \right\}.
\]
Since \(\varpi+\overline{\varpi}\in\varpi^2\calO_{K_v}\), we have
\(\overline{\varpi}/\varpi\equiv-1\pmod{\varpi}\).  Reduction modulo
\(\varpi\) therefore identifies the quotient occurring in the definition of
\(H_3\) with the space of alternating matrices in \(\M_s(k_v)\): its diagonal
entries lie in \(\varpi\calO_v\), and the kernel of reduction is
\(\varpi\Her_s(\calO_{K_v})\).  Hence
\[
  \#\frakB(A,s,t)
  =q_v^{(n-s-t)(n-t)+s(s-1)/2}.
\]

Define
\[
  Q_{st}(S)
  =
  \left\{
  A \in \GL_n(\calO_{K_v}) \backslash \tau^{(n)}_{s,t}(\varpi)
  \;\middle|\;
  \sum_{B\in\frakB(A,s,t)}
  \bfe_\bbA\!\bigl(-\tr(SA^{-1}B)\bigr)\ne0
  \right\}.
\]
Then, if we put \(q_{\pi_i}=N_{F/\bbQ}(\pi_i)\), we have
\begin{align}\label{eq:hecke_ramified_tpi}
  A_{\widetilde{T_i}(\varpi)\cdot f}(\gamma_{ij},S)
   & = q_{\pi_i}^{2r_n+n^2}
  \sum_{\substack{s+t\le n  \\ s\ge i}}
  q_v^{(n-s-t)(n-t)+s(s-1)/2}
  \sum_{A\in Q_{st}(\nu_iS[\bft_0(z_j)])}
  \omega(\pi_{i,\infty}^{-1})
  \nonumber                 \\
   & \qquad \cdot
  \rho(\alpha^{*-1}_{A,\infty},\,^t\!\alpha_{A,\infty}^{-1})
  a\!\left(
  f_{l_i m_{sj}}|_{\rho,\omega}h'_{stj},
  \pi_i^{-1}S[\alpha_A]
  \right).
\end{align}

\subsection{The case where \texorpdfstring{\(v\)}{v} is split}
By the same argument as in Lemma~\ref{lem:decomposition_A}, we have the following lemma.

\begin{lemma}\label{lem:decomposition_split}
  A complete set of representatives of
  \(\GL_{2n}(\calO_v)\backslash T_i(\varpi)\)
  is given by
  \[
    \frakS_i(\varpi)
    =
    \left\{
    \begin{pmatrix}
      \varpi^{k_1} &        & a_{uv}          \\
                   & \ddots &                 \\
                   &        & \varpi^{k_{2n}}
    \end{pmatrix}
    \;\middle|\;
    \begin{array}{l}
      k_1+\cdots+k_{2n}=i,\quad k_j\in\{0,1\},    \\[4pt]
      a_{uv}\in\calO^{(k_v)}_v,\; 1\le u<v\le 2n, \\[4pt]
      a_{uv}=0\ \text{if}\ k_u=k_v=1
    \end{array}
    \right\},
  \]
  where \(\calO^{(0)}_{v}=\{0\}\) and
  \(\calO^{(1)}_{v}\subset\calO_v\) is a fixed
  set of representatives of \(\calO_v/\varpi\calO_v\).
\end{lemma}

For each \(\nu_{i}\in N_0\), choose a representative \(\nu_{l_i}\in N_0\) and
an element \(\pi_i\in F^\times\) such that
\[
  \nu_i^{-1}\varpi
  \in
  \pi_i\,\nu_{l_i}^{-1}
  \Bigl(\prod_{v\in\bfh}\calO_v^\times\Bigr)
  \bbA_\infty^\times.
\]

Put
\[
  \frakA_s(\varpi)
  := \left\{
  A =
  \begin{pmatrix}
    \varpi^{k_1} &        & a_{uv}       \\
                 & \ddots &              \\
                 &        & \varpi^{k_n}
  \end{pmatrix}
  \;\middle|\;
  \begin{array}{l}
    k_1+\cdots+k_n=s,\quad k_j\in\{0,1\},         \\[4pt]
    a_{uv}\in\calO_v^{(k_v)},\quad 1\le u<v\le n, \\[4pt]
    a_{uv}=0\ \text{if }k_u=k_v=1
  \end{array}
  \right\}.
\]
We view \(\frakA_s(\varpi)\) as the set of upper-left \(n\times n\) blocks
of matrices in \(\mathfrak{S}_i(\varpi)\).

Then for each \(A\in\frakA_s(\varpi)\),
\(D\in\frakA_{i-s}(\varpi)\) and representative \(z_j\in Z_0\),
there exist
\(z_{m_{ij}}\in Z_0\) and \(\alpha_{A,D}\in\GL_n(K)\) such that
\[
  \bft_0(z_j)\,(\varpi A^{-1},\,{}^tD)
  \in
  \alpha_{A,D}\,
  \bft_0(z_{m_{ij}}(\nu_i\nu_{l_i}^{-1})^{i-s})\,
  \Bigl(\prod_{v\in\mathbf h}\GL_n(\calO_{K_v})\Bigr)\,
  \GL_n(\bbA_{K,\infty}).
\]
Here \((\varpi A^{-1},\,{}^tD)\) denotes the element of \(\GL_n(\bbA_K)\)
whose \(\mathfrak P\)-component is \(\varpi A^{-1}\),
whose \(\overline{\mathfrak P}\)-component is \({}^tD\),
and whose other components are the identity matrix.

Assume that
\[
  \bft\!\bigl((\nu_i\nu_{l_i}^{-1})^{i-s}\bigr)
  \in y_{is}\,\tilk_{n,0}h'_{is},
\]
with \(y_{is} \in \tilg_n^+\) and \(h'_{is} \in \tilg_{n,\infty}^+\).

Define
\[
  \frakB(A,D)=
  \left\{B
  \;\middle|\;
  \begin{pmatrix}
    A & B \\ 0 & D
  \end{pmatrix}
  \in
  \frakS_i(\varpi)
  \right\}
\]
for each \(A\in\frakA_s(\varpi)\) and
\(D\in\frakA_{i-s}(\varpi)\)
and put
\[
  Q_{s}(S)
  =
  \left\{
  (A, D)\in\frakA_{s}(\varpi)\times\frakA_{i-s}(\varpi)
  \;\middle|\;
  \sum_{B \in \frakB(A,D)}
  \bfe_\bbA\!\bigl(-\tr(SA^{-1}B)\bigr)
  \neq 0
  \right\}.
\]
Then, writing \(q_{\pi_i}=N_{F/\bbQ}(\pi_i)\), we obtain
\begin{align}\label{eq:hecke_split_tpi}
  A_{T_i(\varpi)\cdot f}(\gamma_{ij},S)
   & =q_{\pi_i}^{2r_n+n^2}
  \sum_{s=0}^iq_v^{(n-s)(i-s)}
  \sum_{(A,D)\in Q_s(\nu_iS[\bft_0(z_j)])}
  \omega(\pi_{i,\infty}^{-1})
  \nonumber                \\
   & \quad \cdot
  \rho(\alpha_{A,D,\infty}^{*-1},
  \,{}^t\!\alpha_{A,D,\infty}^{-1})
  a\!\left(
  f_{l_i,m_{ij}}|_{\rho,\omega}h'_{is},
  \pi_i^{-1}S[\alpha_{A,D}]
  \right).
\end{align}

\section{Pullback formula}\label{sec:pullback}
\subsection{Differential operators on Hermitian automorphic forms}\label{subsec:differential}

Let \(n_1,n_2\) be positive integers such that \(n_1\geq n_2\geq 1\) and put \(n = n_1+n_2\).
We diagonally embed \(\hus{n_1}^\bfa \times \hus{n_2}^\bfa \hookrightarrow \hus{n}^\bfa\) and define the diagonal embedding \(\iota_{n_1,n_2}:G_{n_1}\times G_{n_2} \hookrightarrow G_{n}\).

Let \((\rho_s\boxtimes\tau_s,V_s)\) be a representation of \(\frakK_{n_s,\infty}^\bbC\) for \(s=1,2\).
We fix tuples \(\kappa=(\kappa_v)_{\va}, \, \nu=(\nu_v)_{\va}\) of positive integers.

We will consider \(V := V_{1}\otimes V_2\)-valued differential operators \(\bbD\)
on \(\calA_{n}^{\sharp}({\det}^\kappa)\), satisfying Condition (A) below:

\begin{cond}
  For any automorphic form \(f\in \calA_{n}^{\sharp}(\det\nolimits^\kappa\boxtimes\det^\nu)\), we have
  \[\bbD f(\iota_{n_1,n_2}(g_1,g_2))
    \in \calA_{n_1}^{\sharp}(\det\nolimits^\kappa\rho_{1}\boxtimes\det\nolimits^\nu\tau_1)
    \otimes\calA_{n_2}^{\sharp}(\det\nolimits^\kappa\rho_2\boxtimes\det\nolimits^\nu\tau_2)\]
\end{cond}

Condition (A) corresponds to Case (I) in Ibukiyama~\cite{ibukiyama1999differential} for Siegel modular forms
and aligns with the differential operators constructed for several vector-valued cases in Browning \cite{browning2024constructing}.
This type of operator is also treated for Hermitian automorphic forms on \(U(p,q)\) in the work of Eischen--Liu \cite{eischen2024archimedean} (see also \cite{eischen2018differential}, \cite{eischen2020functions}).

A representation-theoretic interpretation of \(\bbD\) in more general settings was provided in \cite{Takeda2025pullback},
while Case (II) of \cite{ibukiyama1999differential} for Hermitian modular forms was studied by Dunn \cite{Dunn2024Rankin}.
The differential operator \(\bbD\) for Hermitian modular forms was explicitly computed in certain cases in \cite{takeda2025differential},
following the methods developed by Ibukiyama \cite{Ibukiyama2014Higher, Ibukiyama2022Differential}.
In what follows, we recall the relevant results from \cite{Takeda2025pullback} and \cite{takeda2025differential}.
In those references, only the case \(\nu=0\) is treated; however, the results extend easily to the case where \(\nu\) is allowed to be non-zero.

We put \(\pi^+=\left(\pi^+_{v,ij}\right)\)
and \(\pi^+_{v,ij}\) acts on \(f\in \calA_{n}^{\sharp}({\det}^\kappa)\) by right derivation.
Let \(P_v(X)\) be a vector-valued polynomial on a space \(\M_n\) of degree \(n\) variable matrices.
We will give the equivalent condition that the differential operator
\(\bbD=P(\pi^+)=\prod_{\va}P_v((\pi^+_{v,ij}))\) satisfies Condition (A).
Let \(L_{n,\kappa}=\bbC[X,Y]\) be the space of polynomials
in the entries of \((n,\kappa)\)-matrices \(X=(x_{ij})\) and \(Y=(y_{ij})\) over \(\bbC\).

\begin{definition}
  If a polynomial \(f(X, Y) \in L_{n,\kappa}\) satisfies
  \[\sum_{s=1}^\kappa \frac{\partial^2f}{\partial X_{is}\partial Y_{j,s}}=0 \text{  for any  } i,j\in\{1,\ldots,n\}, \]
  we say that \(f(X,Y)\) is a pluriharmonic polynomial.
\end{definition}

\begin{proposition}[Corollary~16 and Proposition~17 of \cite{Takeda2025pullback}]
  \label{prop:diff}
  Let \(n_1,n_2\) be positive integers with \(n_1 \ge n_2 \ge 1\), and set
  \(n = n_1 + n_2\).
  Let
  \((\bsk_s,\bsl_s) = (\bsk_{s,v},\bsl_{s,v})_v\)
  be a family of pairs of dominant integral weights satisfying
  \(\ell(\bsk_{s,v}) \le n_s, \ell(\bsl_{s,v}) \le n_s,
  \ell(\bsk_{s,v}) + \ell(\bsl_{s,v}) \le \kappa_v+\nu_v\)
  for \(s=1,2\) and all \(v\).
  Let \(P_v(T)\) be a
  \(\bigl(
  V_{n_1,\bsk_{1,v},\bsl_{1,v}}
  \otimes
  V_{n_2,\bsk_{2,v},\bsl_{2,v}}
  \bigr)\)-valued polynomial
  on the space \(\M_n\) of \(n \times n\) matrices, for each \(v\).
  Define the differential operator
  \[
    \bbD = P(\pi^+)
    = \prod_v P_v\bigl((\pi^+_{v,ij})\bigr).
  \]
  \begin{enumerate}
    \item The differential operator \(\bbD\) satisfies Condition~{\rm(A)} for
          \({\det}^\kappa\boxtimes{\det}^\nu\) and
          \(({\det}^{\kappa}\rho_{n_1,\bsk_1}\boxtimes{\det}^\nu\rho_{n_1,\bsl_1})\boxtimes ({\det}^{\kappa}\rho_{n_2,\bsk_2}\boxtimes{\det}^\nu\rho_{n_2,\bsl_2})\)
          if and only if, for each \(v\), the polynomial \(P_v(T)\) satisfies the
          following conditions.
          \begin{enumerate}
            \item
                  Define
                  \[
                    \widetilde{P_v}(X_1,X_2,Y_1,Y_2)
                    :=
                    P_v\!\left(
                    \begin{pmatrix}
                        X_1\,{}^tY_1 & X_1\,{}^tY_2 \\
                        X_2\,{}^tY_1 & X_2\,{}^tY_2
                      \end{pmatrix}
                    \right),
                  \]
                  where \(X_i,Y_i \in \M_{n_i,\kappa_v+\nu_v}\).
                  Then \(\widetilde{P_v}\) is pluriharmonic in each pair of variables
                  \((X_i,Y_i)\).

            \item
                  For \((A_i,B_i) \in
                  \frakK_{n_i,v}^{\bbC}
                  := \GL_{n_i}(\bbC) \times \GL_{n_i}(\bbC)\), we have
                  \[
                    P_v\!\left(
                    \left(\begin{smallmatrix} A_1 & 0 \\ 0 & A_2
                        \end{smallmatrix}\right)
                    T
                    \left(\begin{smallmatrix} {}^tB_1 & 0 \\ 0 & {}^tB_2
                        \end{smallmatrix}\right)
                    \right)
                    =
                    \bigl(
                    \rho_{n_1,(\bsk_{1,v},\bsl_{1,v})}(A_1,B_1)
                    \otimes
                    \rho_{n_2,(\bsk_{2,v},\bsl_{2,v})}(A_2,B_2)
                    \bigr)
                    P_v(T).
                  \]
          \end{enumerate}

    \item
          There exists a differential operator \(\bbD\) satisfying
          Condition~{\rm(A)} for \({\det}^\kappa\boxtimes{\det}^\nu\) and
          \(({\det}^{\kappa}\rho_{n_1,\bsk_1}\boxtimes{\det}^\nu\rho_{n_1,\bsl_1})\boxtimes ({\det}^{\kappa}\rho_{n_2,\bsk_2}\boxtimes{\det}^\nu\rho_{n_2,\bsl_2})\)
          if and only if
          \(\bsk_1 = \bsl_2\) and \(\bsl_1 = \bsk_2\).
          If such an operator exists, it is unique up to scalar multiplication.
  \end{enumerate}
\end{proposition}
\begin{remark}
  In particular, the differential operators do not depend on the individual choices of \(\kappa\) and \(\nu\), but only on their sum \(\kappa+\nu\).
\end{remark}
\begin{remark}
  The explicit form of the polynomial \(P_v(T)\) is given in \cite{takeda2025differential} for certain specific cases.
\end{remark}

\subsection{Hermitian Eisenstein series}\label{subsec:Eisenstein}

We introduce the Hermitian Eisenstein series according to Shimura
\cite[\S 16.5]{shimura2000arithmeticity}.
We fix tuples \(\kappa=(\kappa_v)_{\va},\, \nu=(\nu_v)_{\va}\) of non-negative integers.

Consider the following subgroups of \(G_{n}\) for \(r\leq n\):
\begin{align*}
  L_{n,r} & =\left\{\left.
  \begin{pmatrix}
    A & 0   & 0            & 0   \\
    0 & I_r & 0            & 0   \\
    0 & 0   & \adj{A}^{-1} & 0   \\
    0 & 0   & 0            & I_r \\
  \end{pmatrix}
  \in G_{n}\right| A \in \GL_{n-r}(K)\right\}, \\
  U_{n,r} & =\left\{
  \begin{pmatrix}
    I_{n-r} & *   & *       & *   \\
    0       & I_r & *       & 0   \\
    0       & 0   & I_{n-r} & 0   \\
    0       & 0   & *       & I_r \\
  \end{pmatrix}
  \in G_{n}\right\},                           \\
  G_{n,r} & =\left\{
  \begin{pmatrix}
    I_{n-r} & 0 & 0       & 0 \\
    0       & * & 0       & * \\
    0       & 0 & I_{n-r} & 0 \\
    0       & * & 0       & * \\
  \end{pmatrix}
  \in G_{n}\right\}.
\end{align*}
Then the subgroups \(P_{n,r}=G_{n,r}L_{n,r}U_{n,r}\) are the standard parabolic subgroups of \(G_{n}\) and there are natural embeddings
\(a_n:F\hookrightarrow A^+_n\), \(t_{n,r}: \GL_{n-r}(K)\hookrightarrow L_{n,r}\) and \(s_{n,r}: G_{r}\hookrightarrow G_{n,r}\).
Define \(G_{n,r,v}, G_{n,r}(\bbA), L_{n,r,v}, L_{n,r}(\bbA)\), etc. in the same way as \(G_{n,v}, G_{n}(\bbA)\), etc.

By the Iwasawa decomposition, \(G_{n}(\bbA)\) (resp. \(G_{n,v}\)) can be decomposed as
\(G_{n}(\bbA)=P_{n,r,\bbA}\frakK_{n,\bbA}\) (resp. \(G_{n,v}=P_{n,r,v}\frakK_{n,v}\)).

We consider an unramified Hecke character \(\chi\) of \(K\)
whose infinite component \(\chi_\infty\) is given by
\begin{equation}\label{eq:infty_character}
  \chi_\infty(x)=\prod_{\va} |x_v|^{-\kappa_v+\nu_v}\, x_v^{\kappa_v-\nu_v}.
\end{equation}
For each place \(v\) of \(F\), we put
\[
  \chi_v=\prod_{w\mid v}\chi_w.
\]
Fix such a Hecke character \(\chi\).

We put
\[
  \delta(g)=\det\!\left(C \bfi_n + D\right)
\]
for
\(
g=
\begin{pmatrix}
  A & B \\
  C & D
\end{pmatrix}\in \rmu_{n,n}.
\)

\begin{definition}
  Let \(s\in\bbC\).
  For \(g=t_{n,0}(A)\mu k \in G_{n,v}\) with \(A\in \GL_n(K_v)\), \(\mu\in U_{n,0,v}\), and \(k\in \tilk_{n,v}\), define
  \[
    \varepsilon_{n,\kappa,\nu,v}(g,s;\chi)=
    \begin{cases}
      |\det(\adj{A}A)|_v^{\,s}\,\chi_v(\det A^*)
       & (v\in\bfh), \\[6pt]
      \det(g)^{\nu_v}|\delta(g)|^{\,\kappa_v+\nu_v-2s}\,\delta(g)^{-\kappa_v-\nu_v}
       & (\va).
    \end{cases}
  \]
  Set
  \[
    \varepsilon_{n,\kappa,\nu}(g,s;\chi)
    =\prod_v \varepsilon_{n,\kappa,\nu,v}(g_v,s;\chi).
  \]
  The Hermitian Eisenstein series on \(G_{n}(\bbA)\) is defined by
  \[
    E_{n,\kappa,\nu}(g,s;\chi)
    =\sum_{\gamma\in P_{n,0}\backslash G_n}
    \varepsilon_{n,\kappa,\nu}(\gamma g,s;\chi).
  \]
  For \(\theta\in \frakK_{n,0}\), we further define
  \[
    E_{n,\kappa,\nu}^{\theta}(g,s;\chi)
    =E_{n,\kappa,\nu}(g\theta^{-1},s;\chi).
  \]
\end{definition}

The Hermitian Eisenstein series \(E_{n,\kappa,\nu}(g,s;\chi)\) and \(E^\theta_{n,\kappa,\nu}(g,s;\chi)\) converge absolutely and locally uniformly for \(\Re(s)>n\)
(see, for example, Shimura~\cite{shimura1997Euler}).

\begin{proposition}[{\cite[Proposition~17.7]{shimura2000arithmeticity}}]
  Let \(\mu\) be a positive integer such that \(\mu\geq n\).
  If \(\kappa_v+\nu_v=\mu\) for every \(\va\), then \(E_{n,\kappa,\nu}(g,\mu/2;\chi)\) belongs to \(\calA^\sharp_n({\det}^\kappa\boxtimes{\det}^\nu)\),
  except when \(\mu=n+1\), \(F=\bbQ\), \(\chi=\chi_{K/F}^{n+1}\),
  where \(\chi_{K/F}\) is the quadratic character associated with the quadratic extension \(K/F\).
\end{proposition}

Let \(\bsk,\bsl\) be families of dominant integral weights such that \(\ell(\bsk_v), \ell(\bsl_v) \leq r\leq n\), and \(\ell(\bsk_v) + \ell(\bsl_v) \leq \kappa_v+\nu_v\) for each infinite place \(v\) of \(F\).
We put \(\rho_r = {\det}^\kappa\rho_{r,\bsk}\boxtimes{\det}^\nu\rho_{r,\bsl}\) and \(\rho_n = {\det}^\kappa\rho_{n,\bsk}\boxtimes{\det}^\nu\rho_{n,\bsl}\).

\begin{definition}
  We define
  \[\varepsilon(f)_{\kappa,\nu,v}^n(g,s ; \chi)=\left\{
    \begin{array}{ll}
      \abs{\det \left(\adj{A_r}A_r\right)}_v^{s}\chi_v(\det A_r^*)f(h_r)                               & (v\in\bfh), \\
      \norm{\delta(g)\delta(h_r)^{-1}}^{\kappa_v+\nu_v-2s}\rho_n(M(g))^{-1}\rho_r(M(h_r))f(h_r) \qquad & (\va )
    \end{array}
    \right.\]
  for \(f\in \calS_{r}(\rho_r)\) (\(r<n\))
  and \(g=t_{n,r}(A_r)\, \mu_r\,  s_{n,r}(h_r)\, k \in G_{n,v}\) with
  \(A_r\in \GL_{n-r}(K_v)\), \(\mu_r\in U_{n,r}\),
  \(h_r\in G_{r,v}\) and \(k\in \frakK_{n,v}\).
  Then, we put
  \[\varepsilon(f)_{\kappa,\nu}^n(g,s;\chi)=\prod_v\varepsilon(f)_{\kappa,\nu,v}^n(g_v,s ;\chi)\]
  and define the Hermitian Klingen-Eisenstein series \(\left[f\right]_{\kappa,\nu}^n(g,s;\chi)\) on
  \(G_{n}(\bbA)\) associated with \(f\) by
  \[\left[f\right]_{\kappa,\nu}^n(g,s; \chi)=\sum_{\gamma\in P_{n,r}\backslash G_n}\varepsilon(f)_{\kappa,\nu}^n(\gamma g,s; \chi).\]
\end{definition}

\subsection{Pullback formula}\label{subsec:pullback}
Let \(n_1\) and \(n_2\) be positive integers with \(n_1 \geq n_2\), and put \(n = n_1 + n_2\).
Let \(m = m_F\) be the number of infinite places of \(F\).
Let \(\kappa = (\kappa_v)_v, \nu = (\nu_v)_v\) be families of positive integers,
such that \(\kappa_v + \nu_v = \mu\) for each infinite place \(v\) of \(F\), where \(\mu\) is the fixed positive integer as above,
and let
\(\bsk = (\bsk_v)_v\) and \(\bsl = (\bsl_v)_v\) be families of dominant integral weights
such that \(\ell(\bsk_v) \leq n_2\), \(\ell(\bsl_v) \leq n_2\), and
\(\ell(\bsk_v) + \ell(\bsl_v) \leq \kappa_v + \nu_v\) for each infinite place \(v\) of \(F\).
Assume moreover that \(\mu>n\).
We put \(\rho_{n_1} = {\det}^{\kappa}\rho_{n_1,\bsk}\boxtimes{\det}^\nu\rho_{n_1,\bsl}\) and
\(\rho'_{n_2} = {\det}^{\kappa}\rho_{n_2,\bsl}\boxtimes{\det}^\nu\rho_{n_2,\bsk}\).
Let \(P_{\bsk,\bsl}^{(n_1,n_2)}(T)\) be the polynomial constructed in \cite[Theorem~5.13]{takeda2025differential},
with scalar weight \(\kappa+\nu\), and set
\begin{equation}\label{eq:differential_operator}
  \dkl^{(n_1,n_2)}
  = P_{\bsk,\bsl}^{(n_1,n_2)}(\pi^+),
\end{equation}
which satisfies condition~{\rm(A)} for \({\det}^{\kappa}\boxtimes{\det}^{\nu}\) and
\(\rho_{n_1}\boxtimes\rho'_{n_2}\).
When no confusion can arise, we simply write \(\dkl=\dkl^{(n_1,n_2)}\).

Let \(\chi_{K/F}\) be the quadratic character
associated with the quadratic extension \(K/F\).
For a Hecke eigenform \(f\) on \(\tilg^+_{n}(\bbA)\) of weight \((\rho,V)\), and a Hecke character \(\eta\) of \(K\), we set
\[
  D(s,f;\eta)
  = L(s-n+\tfrac{1}{2},\, f\otimes \eta,\, \mathrm{St})
  \cdot \left( \prod_{i=0}^{2n-1} L_F(2s-i,\, \eta \cdot \chi_{K/F}^i) \right)^{-1},
\]
where \(L(*, f\otimes \eta, \mathrm{St})\) denotes the standard \(L\)-function of degree \(2n\) attached to \(f\otimes \eta\) as an automorphic form on \(G_{n}(\bbA)\).
Here, \(L_F(*,\eta)\) (resp. \(L_F(*, \eta \cdot \chi_{K/F})\)) is the Hecke \(L\)-function attached to \(\eta\) (resp. \(\eta \cdot \chi_{K/F}\)).

We put
\[
  g^\natural
  =
  \begin{pmatrix} 0 & I_n \\ I_n & 0 \end{pmatrix}
  g
  \begin{pmatrix} 0 & I_n \\ I_n & 0 \end{pmatrix},
  \qquad
  f^\dagger_\chi(g)=\chi(g)\,\overline{f(g^\natural)},
\]
for a Hecke character \(\chi\) of \(K\) satisfying \eqref{eq:infty_character}.
Then, if \(f\in \calA_n^{\sharp}(\det^\kappa\rho\boxtimes\det^\nu\tau)\),
we have
\(  f^\dagger_\chi \in \calA_n^{\sharp}(\det^\kappa\tau\boxtimes\det^\nu\rho)\).

\begin{theorem}\label{thm:pullback}
  Assume that \(\Re(s)>n\).
  Let the above notation and assumptions stand.
  Then, for any Hecke cusp form \(f\in \calS^\sharp_{n_2}(\rho'_{n_2})\), we have
  \[
    \left(
    f,\,
    (\dkl E_{n,\kappa,\nu})(\iota_{n_1,n_2}(g_1,*),\overline{s};\chi)
    \right)
    =
    c(s,\rho'_{n_2})\,
    D(s,f;\overline{\chi})\,
    \overline{[f^\dagger_\chi]_{\kappa,\nu}^{n_1}(g_1,s; \chi)},
  \]
  where \(c(s,\rho'_{n_2})\) is a function depending only on \(\rho'_{n_2}\),
  which is defined in \cite{Takeda2025pullback}.
\end{theorem}
\begin{proof}
  The case where \(\nu=0\) is treated in \cite[Theorem~28]{Takeda2025pullback}, and the general case can be proved by a similar argument.
\end{proof}

The constant \(c(s,\rho'_{n_2})\) in the theorem is generally not explicitly stated, but in certain cases its value can be computed.
We set \(|\bsk|=\sum_{v,i} k_{v,i}\) and \(|\bsl|=\sum_{v,i} l_{v,i}\) for the tuple of dominant integral weights \((\bsk,\bsl)\).
\begin{proposition}\label{prop:c}
  Let the above notions and assumptions stand.
  Assume that \(\kappa_v+\nu_v = \mu\) for all \(\va\).
  Then the constant \(c(\mu/2,\rho'_{n_2})\) appearing in Theorem~\ref{thm:pullback}
  is given by
  \begin{align*}
    c(\mu/2,\rho'_{n_2})
    ={} & 2^{-2mn_2^2 - m(n_2-2)\mu + |\bsk| + |\bsl|}\,
    \pi^{mn_2^2}                                         \\
        & \cdot
    \prod_{\va}
    \Biggl(
    \frac{
      \prod_{i=1}^{\ell(\bsk_v)}
      \bigl(k_{v,i} + \ell(\bsk_v) - i\bigr)!
    }{
      \prod_{1 \le i < j \le \ell(\bsk_v)}
      \bigl(k_{v,i} - k_{v,j} + j - i\bigr)
    }
    \cdot
    \frac{
      \prod_{i=1}^{\ell(\bsl_v)}
      \bigl(l_{v,i} + \ell(\bsl_v) - i\bigr)!
    }{
      \prod_{1 \le i < j \le \ell(\bsl_v)}
      \bigl(l_{v,i} - l_{v,j} + j - i\bigr)
    }                                                    \\
        & \hspace{25mm}
    \cdot
    \prod_{i=1}^{n_2}
    \frac{1}{
      \bigl(k_{v,i} + l_{v,i} + \mu - i\bigr)_{(n_2)}
    }
    \Biggr),
  \end{align*}
  where \((x)_{(r)} = x(x-1)\cdots(x-r+1)\) denotes the descending Pochhammer symbol.
\end{proposition}
\begin{proof}
  The case where \(\nu=0\) is treated in \cite[Theorem~6.6]{takeda2025differential}, and the general case can be proved by a similar argument.
\end{proof}

\section{Integrality of Hermitian automorphic forms}\label{sec:integrality}
Hereafter, we restrict our attention to the case where
\(F = \bbQ\), \(K = \bbQ(\sqrt{-D_K})\) with discriminant \(-D_K\) \((D_K>0)\).
In this setting, finite places are represented by prime numbers \(p\), while there is a unique infinite place, which we denote by \(\infty\).
We also denote \(\chi_{K/F}\) by \(\chi_{-D_K}\).

Since
\[
  \left| \bbQ^\times \backslash \bbA^\times
  / \prod_p\bbZ_p^\times \bbR^\times \right| = 1
\]
in \eqref{eq:representative}, we write the complete set of representatives in \eqref{eq:representative} as
\(\frakG_0=\{\gamma_{11},\ldots,\gamma_{1h_2}\}=\{\gamma_1,\ldots,\gamma_h\}\).

\begin{lemma}
  There is an isomorphism
  \[
    \widetilde{\calA}_n(\rho,\mathbf{1})
    \;\cong\;
    \calA_n(\rho),
    \qquad
    f \;\mapsto\; f|_{G_{n}(\bbA)},
  \]
  where \(\mathbf{1}\) denotes the trivial character of \(\bbR_{>0}\).
\end{lemma}

\begin{proof}
  We set
  \begin{align*}
    \tilgam_s
     & =
    \tilg^+_n(\bbQ)\cap
    \bigl(\gamma_s \tilk_{n,0}\gamma_s^{-1}\cdot \tilg^+_{n,\infty}\bigr) \\
     & =
    G_n(\bbQ)\cap
    \bigl(\gamma_s \frakK_{n,0}\gamma_s^{-1}\cdot G_{n,\infty}\bigr)
  \end{align*}
  for \(\gamma_s\in \frakG_0\).
  Then, by \eqref{eq:representative}, we have
  \begin{align*}
    \tilg^+_{n}(\bbQ)\backslash \tilg^+_{n}(\bbA) / \tilk_{n,0}
     & \;\cong\;
    \bigsqcup_{s=1}^h \tilgam_s \backslash \tilg^+_{n,\infty} \\
     & \;\cong\;
    \bbR_{>0}
    \times \bigl(G_{n}(\bbQ)\backslash G_{n}(\bbA) / \frakK_{n,0}\bigr),
  \end{align*}
  and the assertion follows.
\end{proof}
Henceforth, we identify these spaces via the above isomorphism.

In the following, we further assume that \(\mu\) is a positive even integer, \(\kappa=\nu=\mu/2\), and the Hecke character \(\chi\) is trivial.
For simplicity, we adopt the following notation:
\[
  E_{n,\mu}(g)
  := E_{n,\kappa,\nu}(g,\mu/2;\chi),\qquad
  f^\dagger
  := f^\dagger_{\chi},\qquad
  [f]_\mu^{n}(g)
  := [f]_{\kappa,\nu}^{n}(g,\mu/2;\chi),
\]
for any Hecke eigenform
\( f \in \calA^{\sharp}_r(\rho) \).

We realize the representation space of \(\rho_{n,(\bsk,\bsl)}\) as
\(\bbC[U,V]_{(\bsk,\bsl)}\),
using bideterminants, and define \(R[U,V]_{(\bsk,\bsl)}\) for a commutative
ring \(R\) in the same manner.

For a subring \(R\subset \bbC\), we set
\[
  \calA^{\sharp}_n(\rho_n)(R)
  = \{\, f\in \calA^{\sharp}_n(\rho_n)
  \mid A_f(\gamma_{j},S)
  \in R[U,V]_{(\bsk,\bsl)},\
  \forall\,\gamma_{j},\ S\in\Her_n(K)\,\},
\]
and
\[
  \calS^{\sharp}_{n}(\rho_n)(R)=\calA^{\sharp}_n(\rho_n)(R)\cap \calS_{n}(\rho_n).
\]

To study the integrality properties of Eisenstein series,
we recall results of Shimura~\cite{shimura1997Euler,Shimura1982Confluent}
on the Fourier expansions of Eisenstein series.

Let \(b_p(s;S)\) denote the local Siegel series defined in
\cite[Section~18]{shimura1997Euler} for \(p\in\bfh\), \(s\in\bbC\) and  \(S\in\Her(K_p)_{\geq 0}\).
For \(S\in\Her(K_p)_{\geq 0}\) with \(\det S\neq 0\),
there exists a polynomial \(F_p(X;S)\) in \(X\) with constant term \(1\) and
coefficients in \(\bbZ\) such that
\begin{equation}\label{eq:Siegel_series}
  b_p(s;S)
  =
  \prod_{i=1}^{\lfloor (n+1)/2\rfloor}\bigl(1-p^{-s-2i}\bigr)
  \prod_{i=1}^{\lfloor n/2\rfloor}\bigl(1-\xi_p p^{-s-2i+1}\bigr)\,
  F_p\bigl(p^{-s};S\bigr)
  ,
\end{equation}
where
\[
  \xi_p=
  \begin{cases}
    1  & \text{if \(K_p\simeq \bbQ_p\times \bbQ_p\),}                   \\
    -1 & \text{if \(K_p/\bbQ_p\) is an unramified quadratic extension,} \\
    0  & \text{if \(K_p/\bbQ_p\) is a ramified quadratic extension.}
  \end{cases}
\]
\begin{remark}
  The Siegel series in the inert case is treated in \cite{hironaka1999spherical}.
  The split case is discussed in \cite{sato2005fourier}.
\end{remark}

\begin{proposition}[\texorpdfstring{%
      \cite[Propositions~18.14 and~19.2]{shimura1997Euler}, \cite{Shimura1982Confluent}%
    }{Shimura}]\label{prop:Eisen_Fourier}
  Let \(\mu\) be a positive even integer with \(\mu \ge n\) and \(\mu\neq n+1\).
  For a non-degenerate \(S \in \Her_n(K)\),
  put
  \[
    A_{n,\mu}(g,S):=A_{E_{n,\mu}}(g,S),
  \]
  where the right-hand side denotes the normalized Fourier coefficient in~\eqref{eq:arithmetic_fourier}.
  Then, for \(A\in\GL_n(\bbA_K)\),
  \begin{align*}
    A_{n,\mu}(\bfm(A),S)
    ={} & \epsilon_n\cdot2^n
    (D_K^{\lfloor n/2\rfloor}\det S)^{\mu-n}       \\
        & \quad\cdot
    \prod_{p\in\bfh}|\det(A_pA_p^*)|_p^{\,n-\mu/2} \\
        & \quad\cdot
    \prod_{i=0}^{n-1}L(1-\mu+i,\chi_{-D_K}^{i})^{-1}
    \prod_{p\in\bfc}F_p(p^{-\mu};S[A_p]),
  \end{align*}
  where \(\epsilon_n = \begin{cases} -1 & (n \equiv 2 \pmod 4) \\ 1 & (\text{otherwise}) \end{cases}\).
  Here, \(\bfc\) is the subset of \(\bfh\) consisting of all the finite places \(p\) of the following two types: (i) \(p\) is ramified in \(K\), (ii) \(p\) is unramified in \(K\) and \(\det(\adj{A_p} S A_p)\notin \bbZ_p^\times\).

  In particular,
  \begin{align*}
    A_{n,\mu}(\gamma_{i},S)
    ={} & \epsilon_n\cdot2^n
    (D_K^{\lfloor{n/2}\rfloor}\det S)^{\mu-n}
    \prod_{i=0}^{n-1}
    L(1-\mu+i, \chi_{-D_K}^{i})^{-1} \\
        & \qquad\cdot
    \prod_{p\in\bfh}|\det(z_{i,p}\adj{z_{i,p}})|_p^{\,n-\mu/2}
    \prod_{p\in\bfc}
    F_p\bigl(p^{-\mu}; S[\bft(z_{i,p})]\bigr)
  \end{align*}
  for \(\gamma_i\in \frakG_0\).
\end{proposition}

Here, the local factors \(F_p(X;S)\) satisfy the following functional equation
(cf.~Ikeda~\cite{ikeda2008lifting}):
\[
  F_p(p^{-2n}X^{-1};S)
  = \chi_{-D_K}((-1)^{n/2}\det S)^{s_n}
  (p^n X)^{-v_p(D_K^{\lfloor{n/2}\rfloor}\det S)}F_p(X;S),
\]
where \(s_n=0\) if \(n\) is odd and \(s_n=1\) if \(n\) is even.
Using this functional equation,
we obtain
\begin{align}\label{eq:F_v}
  A_{n,\mu}(\bfm(A),S)
  ={} & 2^n\prod_{i=0}^{n-1}
  L(1-\mu+i, \chi_{-D_K}^{i})^{-1}
  \prod_{p\in\bfh}|\det(A_pA_p^*)|_p^{\mu/2}\nonumber \\
      & \qquad\cdot
  \prod_{p\in\bfc} F_p(p^{\mu-2n};\adj{A_p}SA_p).
\end{align}

\begin{lemma}\label{lem:Fp_int}
  If \(S\in \Her_n(\calO_K)_{>0}\) is a non-degenerate Hermitian matrix, then
  \(F_p(p^{-n}X;S)\in \bbZ[X]\).
  In particular, if \(\mu\ge n\), then \(F_p(p^{\mu-2n};S)\in\bbZ\).
\end{lemma}
\begin{proof}
  This follows easily from Lemma~4.2.3 of Katsurada~\cite{katsurada2015Koecher}.
\end{proof}

We put
\[
  \widetilde{E}_{n,\mu}(g)
  =
  \left(\prod_{i=0}^{n-1}
  L(1-\mu+i, \chi_{-D_K}^{\,i})\right) E_{n,\mu}(g)
  =
  \left(\prod_{i=0}^{n-1}
  -\frac{B_{\mu-i,\chi_{-D_K}^{\,i}}}{\mu-i}\right) E_{n,\mu}(g),
\]
where \(B_{m,\chi}\) denotes the generalized Bernoulli number attached to a Dirichlet character \(\chi\).
By Proposition~\ref{prop:Eisen_Fourier}, \eqref{eq:F_v}, and Lemma~\ref{lem:Fp_int},
we obtain the following result.

\begin{proposition}\label{prop:Eisen_bounded}
  Let \(\mu\) be a positive even integer with \(\mu \ge n\) and \(\mu\neq n+1\).
  Then the Eisenstein series \(E_{n,\mu}(g)\) and \(\widetilde{E}_{n,\mu}(g)\) belong to
  \(\calA_n({\det}^\mu)(\bbQ)\), and the denominators of their Fourier coefficients
  are bounded.

  Moreover, for a prime \(p\) satisfying the following conditions,
  we have \(\widetilde{E}_{n,\mu}(g)\in\calA_n({\det}^\mu)(\bbZ_{(p)})\):
  \begin{itemize}
    \item \(p>\mu+1\);
    \item \(p\) does not divide \(D_K\);
    \item either \(\mu\ge 2n\), or
          \(z_{i,p}\in \calO_{K_p}^{\times}\) for all \(1\le i\le h\).
  \end{itemize}
\end{proposition}

\begin{remark}\label{rem:Bernoulli}
  By the von Staudt-Clausen theorem, if \(p > \mu+1\), then
  \(p\) does not divide the denominator of \(B_{\mu-i}\) for any \(0 \le i \le n-1\).
  On the other hand, if a prime \(p\) does not divide \(D_K\), then
  the generalized Bernoulli number \(B_{m,\chi_D}\) is a \(p\)-integer for every integer \(m\)
  (see Leopoldt~\cite{VonLeopoldt1958Eine}).
\end{remark}

As before, we assume that the dominant integral weights \(\bsk\) and \(\bsl\)
satisfy \(\ell(\bsk)+\ell(\bsl) \le \mu\),
and we set
\(\ell(\bsk, \bsl) = \max\{\ell(\bsk), \ell(\bsl)\}\).
For an integer \(r \ge \ell(\bsk, \bsl)\),
we define
\(\rho_r = {\det}^{\mu/2} \rho_{r,\bsk}\boxtimes{\det}^{\mu/2}\rho_{r,\bsl}\) and
\(\rho'_r = {\det}^{\mu/2} \rho_{r,\bsl}\boxtimes{\det}^{\mu/2}\rho_{r,\bsk}\),
and set \(|\rho_r| = |\bsk| + |\bsl| + r\mu\).
For an element \(v \otimes v' \in V_{\rho_r} \otimes V_{\rho'_r}\),
we often omit the symbol \(\otimes\) and simply write \(vv' = v \otimes v'\).

In order to present the pullback formula for Eisenstein series in a convenient
form, we introduce the following notation.

We set
\begin{align*}
  \bbL_F(m,\mu)
   & =
  \prod_{i=0}^{m-1} L_F(1-\mu+i,\chi_{-D_K}^i), \\
  \bbL(s,f)
   & =D_K^{r/2}\cdot
  \prod_{j=1}^{r}
  \Gamma_\bbC\!\left(s+\tfrac{\mu}{2}+k_j-j+\tfrac12\right)
  \Gamma_\bbC\!\left(s+\tfrac{\mu}{2}+l_j+r-j-\tfrac12\right)
  \frac{L(s,f,\mathrm{St})}{(f,f)},             \\
  \calC_{n,\mu}(f)
   & =
  \frac{\bbL_F(n,\mu)}{\bbL_F(2r,\mu)}\,
  \bbL\bigl((\mu+1)/2-r,f\bigr),
\end{align*}
where \(\Gamma_\bbC(s)=2(2\pi)^{-s}\Gamma(s)\), and
\(f\in\calS_{r}^{\sharp}(\rho_r)\) is a Hecke cusp form.

We remark that, for
\[
  f=(f_1,\ldots,f_h)^\sharp \in \calA_r^{\sharp}(\rho_r),
  \qquad
  f_i(Z)
  =
  \sum_T a\bigl(f_i,T\bigr)\,
  \exp\!\bigl(2\pi\sqrt{-1}\,\tr(TZ)\bigr),
\]
a straightforward computation shows that one can write
\[
  f^\dagger
  =
  (f_1^\dagger,\ldots,f_h^\dagger)^\sharp
  \in
  \calA_r^{\sharp}(\rho'_r),
\]
where the Fourier coefficients are given by
\[
  a(f_i^\dagger,T)
  =
  (\sqrt{-1})^{|\rho_r|}
  \,\overline{a(f_i,T)}.
\]
This description relies on the assumption that the level is \(1\).
If the level is not \(1\), such a simple expression in terms of Fourier
coefficients is no longer available.

Let \(n_1\) and \(n_2\) be positive integers with \(n_1\ge n_2\), and put
\(n=n_1+n_2\).
We define
\[
  \calE_{\mu}(g_1,g_2)
  :=
  \calE^{(n_1,n_2)}_{\mu,\bsk,\bsl}(g_1,g_2)
  =
  (2\pi \sqrt{-1})^{-(|\bsk|+|\bsl|)}\,
  \dkl\,\widetilde{E}_{n_1+n_2,\mu}
  \bigl(\iota_{n_1,n_2}(g_1,g_2)\bigr).
\]

\begin{remark}
  The normalized Eisenstein series satisfies
  \[
    \widetilde{E}_{n,\mu}(g)
    =
    \bbL_F(n,\mu)\,E_{n,\mu}(g).
  \]
\end{remark}

We also define a constant \(M_r\) by
\[
  M_r
  :=
  \max\!\left(
  \left\{
  \mu+k_j-j-r+1,\;
  \mu+l_j-j
  \right\}_{1\le j\le r}
  \cup
  \left\{
  \mu,\;D_K,\;
  k_1+l_1+\mu-1
  \right\}
  \right).
\]

\begin{lemma}\label{lem:pullback}
  Let \(n_1\) and \(n_2\) be positive integers with \(n_1\ge n_2\), and put \(n=n_1+n_2\).
  For \(r\ge \ell(\bsk,\bsl)\), let \(\{f^{(r)}_d \mid d=1,\ldots,e_r\}\) be an orthogonal basis of
  \(\calS_{r}^{\sharp}(\rho_r)\) with respect to the Petersson inner product,
  consisting of Hecke cusp forms.
  Assume that \(\mu \ge n+2\).
  Then we have
  \[
    \calE_{\mu}(g_1,g_2)
    =
    \sum_{r=\ell(\bsk,\bsl)}^{n_2}
    c_r
    \sum_{d=1}^{e_r}
    \overline{\calC_{n,\mu}\bigl(f^{(r)}_d\bigr)}\,
    [f^{(r)}_d]_{\mu}^{n_1}(g_1)\,
    [f^{(r)\dagger}_d]_{\mu}^{n_2}(g_2),
  \]
  where \(c_r\) is a \(p\)-unit rational number for any prime \(p\) satisfying
  \(p>M_r\).
\end{lemma}

\begin{proof}
  This lemma is a direct corollary of Theorem~\ref{thm:pullback}, whose proof follows
  the same strategy as, for example, \cite[Theorem~5.8]{atobe2023harder}.
  The constant \(c_r\) is given explicitly by
  \begin{align*}
    c_r
    ={} &
    (2\pi \sqrt{-1})^{-(|\bsk|+|\bsl|)}\,
    c(\mu/2,\rho'_r)D_K^{-r/2} \\
        & \quad\cdot
    \frac{
      \prod_{i=0}^{2r-1}
      L_F(1-\mu+i,\chi_{-D_K}^i)\,
      L_F(\mu-i,\chi_{-D_K}^i)^{-1}
    }{
      \prod_{j=1}^{r}
      \Gamma_\bbC(\mu+k_j-j-r+1)\,
      \Gamma_\bbC(\mu+l_j-j)
    }.
  \end{align*}
  The \(p\)-unit property of \(c_r\) follows from Proposition~\ref{prop:c}
  and Remark~\ref{rem:Bernoulli}.
\end{proof}

Let \(W_{\calE_\mu}(g_1,g_2;S_1,S_2)\)
denote the adelic bi-Fourier coefficient of \(\calE_\mu\),
obtained by integrating in the two unipotent variables with the character
\((X_1,X_2)\mapsto\bfe_\bbA(\tr(S_1X_1)+\tr(S_2X_2))\).
We define the normalized bi-Fourier coefficient by
\begin{align*}
  \epsilon_{\mu,\bsk,\bsl}^{(n_1,n_2)}(g_1,g_2;S_1,S_2)
  ={} 
  \bigl(
  \eta_{\rho_{n_1}}(g_{1,\infty},S_1)^{-1}
  \otimes
  \eta_{\rho'_{n_2}}(g_{2,\infty},S_2)^{-1}
  \bigr) 
  W_{\calE_\mu}(g_1,g_2;S_1,S_2),
\end{align*}
and abbreviate it to \(\epsilon_\mu(g_1,g_2;S_1,S_2)\).

Applying the pullback formula and taking normalized Fourier coefficients in the two variables gives
\begin{equation}\label{eq:epsilon_mu}
  \epsilon_\mu(g_1,g_2;S_1,S_2)
  =
  \sum_{r=\ell(\bsk,\bsl)}^{n_2}
  (\sqrt{-1})^{|\rho_r|}c_r
  \sum_{d=1}^{e_r}
  \overline{\calC_{n,\mu}\bigl(f^{(r)}_d\bigr)}\,
  A_{[f^{(r)}_d]_{\mu}^{n_1}}(g_1;S_1)\,
  \overline{A_{[f^{(r)}_d]_{\mu}^{n_2}}(g_2;S_2)}.
\end{equation}

We put
\begin{align*}
  \calF_\mu(g_1,\gamma^{(n_2)};S_2)
   & =\calF^{(n_1,n_2)}_\mu(g_1,\gamma^{(n_2)};S_2)  \\
   & =(\sqrt{-1})^{-(|\bsk|+|\bsl|)}
  \sum_{S_1\in\Her_{n_1}(K)}
  \bigl(\eta_{\rho_{n_1}}(g_{1,\infty},S_1)\otimes 1\bigr)
  \epsilon_\mu(g_1,\gamma^{(n_2)};S_1,S_2)           \\
   & =\sum_{r=\ell(\bsk,\bsl)}^{n_2}(-1)^{r\mu/2}c_r
  \sum_{d=1}^{e_r}\overline{\calC_{n,\mu}(f^{(r)}_d)}\cdot
  \overline{A_{[f^{(r)}_d]_{\mu}^{n_2}}(\gamma^{(n_2)},S_2)}
  [f^{(r)}_d]_{\mu}^{n_1}(g_1),                      \\
  \calG_\mu(\gamma^{(n_1)},g_2;S_1)
   & =\calG^{(n_1,n_2)}_\mu(\gamma^{(n_1)},g_2;S_1)  \\
   & =(\sqrt{-1})^{-(|\bsk|+|\bsl|)}
  \sum_{S_2\in\Her_{n_2}(K)}
  \bigl(1\otimes\eta_{\rho'_{n_2}}(g_{2,\infty},S_2)\bigr)
  \epsilon_\mu(\gamma^{(n_1)},g_2;S_1,S_2)           \\
   & =(\sqrt{-1})^{-(|\bsk|+|\bsl|)}
  \sum_{r=\ell(\bsk,\bsl)}^{n_2}c_r
  \sum_{d=1}^{e_r}\overline{\calC_{n,\mu}(f^{(r)}_d)}\cdot
  A_{[f^{(r)}_d]_{\mu}^{n_1}}(\gamma^{(n_1)},S_1)
  [f^{(r)\dagger}_d]_{\mu}^{n_2}(g_2)
\end{align*}
for \(\gamma^{(n_i)} \in \frakG^{(n_i)}_0\) and \(S_i \in \Her_{n_i}(K)\),
where \(i=1,2\).

\begin{proposition}\label{prop:integrality_pullback}
  We have
  \[
    (\sqrt{-1})^{|\rho_{n_2}|}\,\epsilon_\mu(\gamma_{(n_1)},\gamma_{(n_2)};S_1,S_2)
    \in (V_{\rho_{n_1}}\otimes V_{\rho'_{n_2}})(\bbQ)
  \]
  for any Hermitian matrices \(S_i\in \Her_{n_i}(K)\) and
  \(\gamma_{(n_1)}\in \frakG^{(n_1)}_0\), \(\gamma_{(n_2)}\in \frakG^{(n_2)}_0\).

  More precisely,
  \[
    (\sqrt{-1})^{|\rho_{n_2}|}\,\epsilon_\mu(\gamma_{(n_1)},\gamma_{(n_2)};S_1,S_2)
    \in (V_{\rho_{n_1}}\otimes V_{\rho'_{n_2}})(\bbZ_{(p)}),
  \]
  if a prime \(p\) satisfies the following conditions:
  \begin{itemize}
    \item \(p>\mu+1\).
    \item \(p\) does not divide \(D_K\).
    \item The polynomial \(P_{\bsk,\bsl}^{(n_1,n_2)}(T)\) appearing in
          \eqref{eq:differential_operator} has coefficients in
          \((V_{\rho_{n_1}}\otimes V_{\rho'_{n_2}})(\bbZ_{(p)})\).
    \item Either \(\mu\ge 2n\), or
          \(z_{i,p}\in \calO_{K_p}^{\times}\) for all \(1\le i\le h\).
  \end{itemize}

  In particular, under the above assumptions, we have
  \begin{align*}
    \calF_\mu(g_1,\gamma_{(n_2)};S_2)
     & \in
    (\calA_{n_1}^{\sharp}(\rho_{n_1})\otimes
    V_{\rho'_{n_2}})(\bbZ_{(p)}), \\
    \calG_\mu(\gamma_{(n_1)},g_2;S_1)
     & \in
    (V_{\rho_{n_1}}\otimes
    \calA_{n_2}^{\sharp}(\rho'_{n_2}))(\bbZ_{(p)}).
  \end{align*}
\end{proposition}
\begin{proof}
  For any
  \(f=(f_1,\ldots,f_h)^\sharp\in\calA_n^\sharp({\det}^{\mu})\), we have
  \[
    2^{-|\bsk|-|\bsl|}\dkl f
    =
    \bigl(
    P_{\bsk,\bsl}^{(n_1,n_2)}(\partial_Z)f_1,\ldots,
    P_{\bsk,\bsl}^{(n_1,n_2)}(\partial_Z)f_h
    \bigr)^\sharp,
  \]
  where \(\partial_Z=(\partial/\partial Z_{ij})\).
  Noting that \(P_{\bsk,\bsl}^{(n_1,n_2)}\) is homogeneous of degree
  \(|\bsk|+|\bsl|\),
  this implies that the action of \((2\pi\sqrt{-1})^{-(|\bsk|+|\bsl|)}\dkl\)
  preserves the rationality properties of the normalized Fourier coefficients of \(f\).
  (See also \cite[Proposition~6.6]{atobe2023harder}.)
\end{proof}

The following lemma is almost obvious from the definition.
\begin{lemma}\label{lem:neq0}
  Let \(F_1,\ldots,F_s\) be linearly independent elements of
  \(\calA^{\sharp}_n(\rho_n)\) consisting of Hecke eigenforms.
  Then there exist
  \[
    C=(C_v)\in \prod_{\va}\M_{\ell(\bsk_v),n}(\bbZ),\qquad
    D=(D_v)\in \prod_{\va}\M_{\ell(\bsl_v),n}(\bbZ),
  \]
  elements \(S_1,\dots,S_s \in \Her_n(K)\), and
  \(h_1,\dots,h_s \in G_n(\bbA_0)\)
  such that
  \[
    \Delta
    := \det\bigl( A_{F_i}(h_j,S_j)(C,D) \bigr)_{1\le i,j\le s}
    \neq 0.
  \]
\end{lemma}
\begin{remark}
  The elements \(h_1,\ldots,h_s\) in the lemma might be chosen from \(\frakG_0\).
\end{remark}

\begin{theorem}\label{thm:integrality_automorphic}
  Let \(\mu\) be a positive even integer with \(\mu>2n\).
  If \(F\) is a Hecke eigenform in \(\calA_n^{\sharp}(\rho_n)\), then
  \[
    F\in \calA_n^{\sharp}(\rho_n)(\bbZ) \otimes_{\bbZ} \bbC.
  \]
  In particular,
  \[
    \calA_n^{\sharp}(\rho_n)= \calA_n^{\sharp}(\rho_n)(\bbZ) \otimes_{\bbZ} \bbC.
  \]
\end{theorem}
\begin{proof}
  This proof is based on an idea suggested by Professor Katsurada.
  Let \(\iota_{n,n} \colon G_{n} \times G_{n} \to G_{2n}\) be the diagonal embedding.
  For each \(S\in{\Lambda_n}_{\ge0}\),
  take the normalized \(S\)-th Fourier coefficient of \(\calE_\mu(g_1,g_2)\) in the second variable at \(g_{2,0}=\gamma_j\).
  By the definition above, this coefficient is \(\calF_\mu(g_1,\gamma_j;S)\), up to the fixed normalization factor \((\sqrt{-1})^{-(|\bsk|+|\bsl|)}\).
  The corresponding adelic Fourier expansion contains the archimedean factor \(\eta_{\rho'_n}(g_{2,\infty},S)\), as in~\eqref{eq:fourier_normalized_expansion}.

  Since Proposition~\ref{prop:Eisen_bounded} holds and \((2\pi \sqrt{-1})^{-(|\bsk|+|\bsl|)}\dkl\) preserves rationality of the normalized Fourier coefficients, it follows that
  \[
    \calF_\mu(g_1,\gamma_j;S)(U_2,V_2)
    \in
    \calA^{\sharp}_{n}(\rho_{n})(\bbZ)\otimes_{\bbZ}\bbQ,
  \]
  when viewed as a function of \(g_1\).
  Although \(\calF_\mu\) may be regarded as a polynomial in \(U_1,V_1,U_2,V_2\),
  we write \(\calF_\mu\) to emphasize that it is a polynomial in \(U_2,V_2\).

  With the notation of Lemma~\ref{lem:pullback},
  set
  \[
    \{F_1,\ldots,F_s\}
    =
    \left\{
    [f^{(r)}_d]^n_\mu
    \mid
    \ell(\bsk,\bsl)\le r\le n,\
    1\le d\le e_r
    \right\}.
  \]
  This is a basis of \(\calA_n^\sharp(\rho_n)\) consisting of Hecke eigenforms.
  Applying Lemma~\ref{lem:neq0} to this basis,
  we take
  \(C\), \(D\), \(S_j\), and \(h_j\) such that \(\Delta\ne0\).

  For \(1\le i\le s\), write \(F_i=[f_{d_i}^{(r_i)}]_\mu^n\)
  and put \(b_i:=(-1)^{r_i\mu/2}c_{r_i}\,\overline{\calC_{2n,\mu}\bigl(f_{d_i}^{(r_i)}\bigr)}\).
  The pullback formula gives
  \[
    \calF_\mu(g_1,h_j;S_j)(C,D)
    =
    \sum_{i=1}^{s}
    b_i\,
    \overline{
      A_{F_i}(h_j,S_j)(C,D)
    }\,
    F_i(g_1).
  \]
  Since
  \[
    \det\left(
    \overline{
      A_{F_i}(h_j,S_j)(C,D)
    }
    \right)_{1\le j,i\le s}
    =
    \overline{\Delta}
    \ne0,
  \]
  we have
  \begin{align*}
    b_1F_1(g_1)
     & =
    \overline{\Delta}^{-1}
    \begin{vmatrix}
      \calF_\mu(g_1,h_1;S_1)(C,D) & \overline{A_{F_2}(h_1,S_1)(C,D)} & \cdots & \overline{A_{F_s}(h_1,S_1)(C,D)} \\
      \vdots                      & \vdots                           & \ddots & \vdots                           \\
      \calF_\mu(g_1,h_s;S_s)(C,D) & \overline{A_{F_2}(h_s,S_s)(C,D)} & \cdots & \overline{A_{F_s}(h_s,S_s)(C,D)}
    \end{vmatrix}.
  \end{align*}
  For \(\mu>2n\), the standard \(L\)-value occurring in
  \(\calC_{2n,\mu}(f_{d_i}^{(r_i)})\) lies in its region of absolute
  convergence, while
  \(L_F(1-\mu+j,\chi_{-D_K}^{\,j})\ne0\) for \(0\le j\le2r_i-1\),
  because
  \(\chi_{-D_K}^{\,j}(-1)=(-1)^{\mu-j}\).
  Proposition~\ref{prop:c} also gives \(c_{r_i}\neq0\).
  Hence \(b_i\neq0\) for every \(i\), and in
  particular
  \[
    F_1
    \in
    \calA_n^{\sharp}(\rho_n)(\bbZ)
    \otimes_{\bbZ}\bbC.
  \]

  Since the Hecke eigenforms \(F_1,\ldots,F_s\) span \(\calA_n^{\sharp}(\rho_n)\), the second assertion follows.
\end{proof}

For a Hecke eigenform \(f\) and \(T \in \calH_n\),
we denote by \(\lambda_f(T)\) the Hecke eigenvalue of \(T\) with respect to \(f\), that is,
\[
  T \cdot f = \lambda_f(T)\, f.
\]
We define the field \(\bbQ(f)\) generated over \(\bbQ\) by all Hecke eigenvalues of \(f\) by
\[
  \bbQ(f) = \bbQ(\{\lambda_f(T) \mid T \in \calH_n\}),
\]
and call it the \emph{Hecke field} of \(f\).
It is well known that \(\bbQ(f)\) is a totally real field or a CM field.

For \(f \in \mathcal A_r^{\sharp}(\rho_r)\) and \(\sigma \in \Aut(\bbC)\), we define
\(f^\sigma := (f_1^\sigma,\ldots,f_h^\sigma)^\sharp\),
where each component is given by
\[
  f_i^\sigma(Z)
  =
  \sum_{T \in {\Lambda_r}_{\ge 0}}
  \sigma(a(f_i,T))
  \exp(2\pi \sqrt{-1}\,\tr(TZ)).
\]
Here, \(\sigma\) acts on \(\bbC[U,V]_{(\bsk,\bsl)}\) by applying \(\sigma\) to each
coefficient.

The rationality of the spaces of automorphic forms and of
Hermitian Klingen-Eisenstein series can be obtained by the same method as that in Appendix~A of Mizumoto~\cite{mizumoto1991poles}.

\begin{proposition}\label{prop:rationality}
  Let \(n \ge r\) be positive integers such that \(\mu \ge n+r+2\).
  Let \(0 \ne f \in \calS^{\sharp}_r(\rho_r)(\bbQ(f))\) be a Hecke cusp form.
  Then the following assertions hold:
  \begin{enumerate}
    \item There exists an orthogonal basis \(\{f_1,\ldots,f_d\}\) of
          \(\calS^{\sharp}_{r}(\rho_r)\)
          consisting of Hecke cusp forms such that
          \(f_j \in \calS^{\sharp}_{r}(\rho_r)(\bbQ(f_j))\) for each \(j\).
    \item We have
          \[
            \left(\calC_{n,\mu}(f)\right)^\sigma
            = \calC_{n,\mu}(f^\sigma)
            \quad \text{for all } \sigma \in \Aut(\bbC).
          \]
          In particular, \[\calC_{n,\mu}(f) \in \bbQ(f).\]
    \item We have
          \[
            ([f]_\mu^n)^\sigma = [f^\sigma]_\mu^n
            \quad \text{for all } \sigma \in \Aut(\bbC).
          \]
          In particular, \[[f]_\mu^n \in \calA^{\sharp}_n(\rho_n)(\bbQ(f)).\]
  \end{enumerate}
\end{proposition}

\begin{corollary}
  Let \(n\) be a positive integer, and let \((\bsk,\bsl)\) be dominant integral
  weights such that \(k_n+ l_n \ge 2n\).
  For any Hecke cusp form
  \(f \in \calS^{\sharp}_{n}\bigl(\rho_{n,(\bsk,\bsl)}\bigr)\)
  and any integer \(s_0\) satisfying
  \(1 \le s_0 \le (k_n+l_n)/2 - n\), we have
  \[
    \frac{L\bigl(s_0 + \frac12,\, f,\, \mathrm{St}\bigr)}
    {D_K^{n/2}\cdot \pi^{|\bsk| + |\bsl| + 2n(s_0+1/2)-n}\cdot (f,f)}
    \in \bbQ(f).
  \]
\end{corollary}

\begin{remark}
  More precise results on the algebraicity of special values of standard
  \(L\)-functions for scalar-valued Hermitian modular forms
  were obtained by Shimura~\cite{shimura2000arithmeticity}.
  See also Bouganis~\cite{bouganis2015algebraicity}.
\end{remark}

\begin{remark}
  When \(n=1\) and the class number of \(K\) is \(1\), we have
  \[
    L\bigl(s_0 + \tfrac{1}{2},\, f^\sharp,\, \mathrm{St}\bigr)
    = L\bigl(s_0 + \tfrac{k}{2},\, f\bigr) L\bigl(s_0 + \tfrac{k}{2},\, f \otimes \chi_K\bigr),
  \]
  where \(f \in S_k(\mathrm{SL}_2(\bbZ))\) is a classical Hecke cusp form and \(L(s, f)\) denotes its Dirichlet \(L\)-series:
  \[
    L(s,f) = \sum_{n=1}^\infty a(n, f) n^{-s} = \prod_p (1 - a(p, f)p^{-s} + p^{k-1-2s})^{-1}.
  \]
  Moreover, the Petersson norm \((f^\sharp, f^\sharp)\) on the Hermitian upper half-plane coincides with the classical one (cf. \cite[Appendix to \S10]{ichino2010periods}).
  Thus, our result is consistent with the period relations shown by Shimura \cite{shimura1977periods}.
\end{remark}

We set \(\calH_n = \otimes_p^{\!\prime} \calH_{n,p}\) to be the (global) Hecke algebra,
and put \(\calH_n(\bbQ) = \calH_n \otimes_{\bbZ} \bbQ\).

We say that an element \(T \in \calH_n(\bbQ)\) is \emph{integral with respect to}
\(\calA^\sharp_n(\rho)(\bbZ)\) if
\[
  T \cdot f \in \calA^\sharp_n(\rho)(\bbZ)
  \quad \text{for all } f \in \calA^\sharp_n(\rho)(\bbZ).
\]
We denote by \(\calH_n^\rho\) the subset of \(\calH_n(\bbQ)\) consisting of
all elements that are integral with respect to
\(\calA^\sharp_n(\rho)\).

The relations defining the completed operators at inert and ramified places are unitriangular over \(\bbZ\), so the original generators are \(\bbZ\)-linear combinations
of the completed operators.
Applying \eqref{eq:hecke_inert_tp}, \eqref{eq:hecke_inert_tpi}, \eqref{eq:hecke_ramified_tpi}, and \eqref{eq:hecke_split_tpi}, we see that
\begin{proposition}
  Under the notation of Lemma~\ref{lem:hecke}, if \(k_n + l_n \ge 2n\), then
  \[
    \underset{p:\text{inert or ramified}}{\bigotimes\nolimits^{\prime}} \calH_{n,p}
    \subset \calH_n^{\rho_{n,(\bsk,\bsl)}}.
  \]
  If, in addition, \(k_n\ge n\) and \(l_n\ge n\), then
  \[
    \calH_n \subset \calH_n^{\rho_{n,(\bsk,\bsl)}}.
  \]
  No assumption on the class number of \(K\) is required.
  In particular, for \(\rho_n={\det}^{\mu/2}\rho_{n,\bsk}\boxtimes{\det}^{\mu/2}\rho_{n,\bsl}\) with \(\mu>2n\), we have
  \[
    \calH_n \subset \calH_n^{\rho_n}.
  \]
\end{proposition}
\begin{lemma}
  Assume that \(\mu>2n\).
  Let \(f\in\calS_n^\sharp(\rho_n)\) be a Hecke eigenform.
  Then \(\lambda_f(T)\) belongs to \(\calO_{\bbQ(f)}\) for any
  \(T\in\calH_n^{\rho_n}\).
\end{lemma}

\begin{proof}
  This follows immediately from Theorem~\ref{thm:integrality_automorphic}
  and the definition of \(\calH_n^{\rho_n}\).
\end{proof}
Now we define the integral ideal \(\frakA(f)\) of \(\bbQ(f)\)
attached to a Hecke cusp form
\(f \in \calS^\sharp_{n}(\rho_n)(\bbQ(f))\),
and state the following integrality lemma, formulated along the same lines as in
Mizumoto~\cite{mizumoto2005congruences,mizumoto1996on,mizumoto1996corrections}.

Put \(V=\bigoplus_{\tau}\bbC f^\tau\), where \(\tau\) runs over all embeddings
of \(\bbQ(f)\) into \(\bbC\), and let \(V^\perp\) be the orthogonal
complement of \(V\) in \(\calS^\sharp_{n}(\rho_n)\).
Let \(\nu(f)\) (resp. \(\kappa(f)\)) denote the exponent of the finite abelian group
\[\calS^\sharp_{n}(\rho_n)(\bbZ)/
  (V(\bbZ)\oplus V^\perp(\bbZ))
  \qquad
  \text{(resp. \(\mathcal{O}_{\bbQ(f)}/
    \bbZ[\lambda_f(T)\mid T\in \calH_n^{\rho_n}]\))}.
\]
We define
\[
  \frakA(f)=\kappa(f)\nu(f)\mathfrak{d}(\bbQ(f)),
\]
where \(\mathfrak{d}(\bbQ(f))\) denotes the different of
\(\bbQ(f)/\bbQ\).

Let \(f\in\calS_n^\sharp(\rho_n)(E)\).
A Fourier coefficient \(A_f(\gamma_j,S)\), with \(\gamma_j\in\frakG_0\) and \(S\in\Her_n(K)_{>0}\),
is called \emph{\(\frakp\)-primitive} if \(v_{\frakp}\bigl(A_f(\gamma_j,S)\bigr)=0\).
We say that \(f\) has a \(\frakp\)-primitive Fourier coefficient if it has a Fourier coefficient with this property.

For a prime ideal \(\frakq\) of \(E\), we define
\[
  v_{\frakq}(f)=\min_{\gamma_j,S}v_{\frakq}(A_f(\gamma_j,S)),
\]
where the minimum is well-defined because the Fourier coefficients of \(f\) have bounded denominators.
We say that \(f\) is \emph{\(p\)-normalized} if \(v_{\frakq}(f)=0\) for every prime ideal \(\frakq\mid p\) of \(E\).
Every nonzero form in \(\calS_n^\sharp(\rho_n)(E)\) can be made \(p\)-normalized by multiplication by an element of \(E^\times\).

\begin{lemma}[Integrality lemma]\label{lem:integrality}
  Assume that \(\mu>2n\).
  Let \(f\in \calS^\sharp_{n}(\rho_n)(\bbQ(f))\) be a Hecke cusp form,
  let \(\frakp\) be a prime ideal of \(\bbQ(f)\),
  and assume that \(f\) has a \(\frakp\)-primitive Fourier coefficient.
  Let \(L\) be an algebraic number field and let \(\frakP\) be a prime ideal of \(L\cdot \bbQ(f)\) lying above \(\frakp\).
  Then, for every \(g\in\calS_n^\sharp(\rho_n)(\calO_L)\),
  we have
  \[
    \frac{(f,g)}{(f,f)}
    \in
    \frakA(f)^{-1}\calO_{L\cdot\bbQ(f),\frakP}.
  \]
  In particular, if \(f\) is \(p\)-normalized for a rational prime \(p\),
  then
  \[
    \frac{(f,g)}{(f,f)}
    \in
    \frakA(f)^{-1}\calO_{L\cdot\bbQ(f),(p)},
  \]
  where \(\calO_{L\cdot\bbQ(f),(p)}\) denotes the localization of \(\calO_{L\cdot\bbQ(f)}\) at the set of prime ideals above \(p\).
\end{lemma}

\begin{proof}
  The proof is identical to that of
  \cite[Theorem~4.1]{mizumoto1996on}
  and \cite[Lemma~3.1]{mizumoto2005congruences}.
\end{proof}

For non-negative integers \(\nu\) and \(m\) with
\(\ell(\bsk,\bsl)\le m<\nu\), let \(e_{(\bsk,\bsl)}(\nu,m)\) denote the exponent of
the finite abelian group
\([\calS^\sharp_{m}(\rho_m)(\bbZ)]_\mu^\nu/
[\calS^\sharp_{m}(\rho_m)]_\mu^\nu(\bbZ)\).
We agree that \(e_{(\bsk,\bsl)}(\nu,m)=1\) if
\(\calS^\sharp_{m}(\rho_m)=\{0\}\), and we put
\[
  \eta_{(\bsk,\bsl)}(\nu)
  :=\prod_{m=\ell(\bsk,\bsl)}^{\nu-1} e_{(\bsk,\bsl)}(\nu,m).
\]

\begin{proposition}\label{prop:integrality_Klingen}
  Assume that \(\mu>2n\).
  Let \(f\in \calS^\sharp_{r}(\rho_r)(\bbQ(f))\ (\ell(\bsk,\bsl)\leq r\leq n)\) be a Hecke cusp form
  having a primitive Fourier coefficient.
  If a prime \(p\) satisfies the following conditions:
  \begin{itemize}
    \item \(p>\mu+1\),
    \item \(p\) does not divide \(D_K\),
    \item the polynomial \(P_{\bsk,\bsl}^{(n,r)}(T)\) appearing in
          \eqref{eq:differential_operator} has coefficients in
          \((V_{\rho_n}\otimes V_{\rho'_r})(\bbZ_{(p)})\),
    \item either \(\mu\ge 2(n+r)\), or
          \(z_{i,p}\in \calO_{K_p}^{\times}\) for all \(1\leq i\leq h\),
  \end{itemize}
  then we have
  \[
    c_r \overline{\calC_{n+r,\mu}(f)}
    A_{[f]_{\mu}^{{n}}}(\gamma,S)
    \in
    V_{\rho_{n}}\bigl(
    \eta_{(\bsk,\bsl)}(r)^{-1}
    \frakA(f)^{-1}
    \calO_{\bbQ(f),(p)}
    \bigr)
  \]
  for any \(\gamma\in \frakG_0\) and \(S\in \Her_n(K)_{\ge 0}\).
\end{proposition}

\begin{proof}
  We consider
  \(\pr(\calG^{(n,r)}_\mu(\gamma,g_2;S))\), where
  \(\pr:\calA^\sharp_n(\rho_n)\to\calS^\sharp_{n}(\rho_n)\)
  denotes the projection onto the space of cusp forms.
  We have
  \[
    \pr(\calG^{(n,r)}_\mu(\gamma,g_2;S))
    =(\sqrt{-1})^{-(|\bsk|+|\bsl|)}c_r
    \sum_{d=1}^{e_r}
    \overline{\calC_{n+r,\mu}\bigl(f^{(r)}_d\bigr)}
    A_{[f^{(r)}_{d}]_{\mu}^{{n}}}(\gamma,S)
    f^{(r)\dagger}_{d}(g_2),
  \]
  and by Proposition~\ref{prop:integrality_pullback},
  \(\pr(\calG^{(n,r)}_\mu(\gamma,g_2;S))\) belongs to
  \((V_{\rho_{n}}\otimes \calA_{r}^{\sharp}(\rho'_{r}))
  (\eta_{(\bsk,\bsl)}(r)^{-1}\bbZ_{(p)})\)
  (cf.~\cite[Proposition~6.2]{mizumoto1996on}).
  Applying Lemma~\ref{lem:integrality} to \(f^\dagger\) and
  \(\pr(\calG^{(n,r)}_\mu(\gamma,g_2;S))\), and using \(\frakA(f^\dagger)=\frakA(f)\),
  we obtain the desired result.
\end{proof}

\section{Main theorem}\label{sec:main}
The notation and assumptions in the previous sections remain in force.

\begin{definition}
  Let \(F,G\in\calA_n^\sharp(\rho_n)\) be Hecke eigenforms, and let \(\frakp\) be a prime ideal of
  \(\bbQ(F,G)=\bbQ(F)\cdot\bbQ(G)\).
  If
  \[
    \lambda_F(T)\equiv \lambda_G(T)\pmod{\frakp}
    \quad\text{for all }T\in\calH_n^{\rho_n},
  \]
  we say that \(F\) and \(G\) are \emph{Hecke congruent}, and we denote this by
  \[
    F \equiv_{ev} G \pmod{\frakp}.
  \]
\end{definition}

For \(r \ge \ell(\bsk,\bsl)\), let
\(\{f^{(r)}_1,\ldots,f^{(r)}_{e_r}\}\) be an orthogonal basis of
\(\calS_{r}^{\sharp}(\rho_r)\) with respect to the Petersson inner product,
consisting of Hecke cusp forms and satisfying the following conditions:
\begin{itemize}
  \item Each \(f^{(r)}_d\) belongs to
        \(\calS_{r}^{\sharp}(\rho_r)(\bbQ(f^{(r)}_d))\).
  \item The basis is permuted under the action of \(\Aut(\bbC)\).
\end{itemize}
The existence of such a basis follows from Proposition~\ref{prop:rationality}.
Whenever a rational prime \(p\) is fixed,
we assume that all Hecke cusp forms under consideration have been made \(p\)-normalized by multiplying them by suitable nonzero scalars.

We are now in a position to state our main result.
\begin{theorem}\label{thm:main}
  Let \(n\) be a positive integer, let \(\mu\) be a positive even integer with
  \(\mu > 2n\), and let \((\bsk,\bsl)\) be a pair of dominant integral weights such
  that \(\ell(\bsk), \ell(\bsl) \le n\).
  Let \(f \in \calS_{r}^{\sharp}(\rho_r)(\bbQ(f))\)
  be a Hecke cusp form which is a constant multiple of \(f^{(r)}_1\)
  for some integer \(r\) satisfying \(\ell(\bsk,\bsl) \le r < n\).
  Let \(\frakp\) be a prime ideal of \(\bbQ(f)\), and let \(p\) be the rational prime below \(\frakp\).
  Assume that \(S_0\in \Her_n(K)_{>0}\) and \(\gamma_0\in \frakG_0\) satisfy the following conditions:
  \begin{enumerate}
    \item\label{item:valuation_condition}
          We have
          \[
            v_\frakp\!\left(
            \overline{\calC_{2n,\mu}\bigl(f\bigr)}
            A_{[f]_{\mu}^{n}}(\gamma_0,S_0)\overline{A_{[f]_{\mu}^{n}}(\gamma_0,S_0)}
            \right)=:-\alpha<0.
          \]
          Note that the above quantity remains unchanged if we replace \(f\) by
          \(cf\) for any constant \(c\in \bbQ(f)^\times\).

    \item\label{item:CalC_condition}
          If \(\ell(\bsk,\bsl)\le \nu \le n-1\), \(1\le d \le e_{\nu}\), and
          \((\nu, d) \neq (r,1)\), then
          \[
            v_\frakq\!\left(\calC_{n+\nu,\mu}(f^{(\nu)}_d)\right)\le 0
          \]
          for any prime ideal \(\frakq\) of \(\bbQ(f,f^{(\nu)}_d)\) lying above \(\frakp\).

    \item\label{item:frakA_condition}
          The ideals \(\frakA(f^{(\nu)}_d)\) are coprime to \(\frakp\) in
          \(\bbQ(f,f^{(\nu)}_d)\) for all
          \(\ell(\bsk,\bsl)\le \nu \le n-1\) and \(1\le d \le e_{\nu}\).

    \item\label{item:prime_conditions}
          The prime \(p\) satisfies the following conditions:
          \begin{enumerate}
            \item\label{item:prime_mu} \(p>\mu+1\);
            \item\label{item:prime_DK} \(p\) does not divide \(D_K\);
            \item\label{item:prime_eta} \(p\) does not divide
                  \(\prod_{\nu=\ell(\bsk,\bsl)}^{n-1}
                  \eta_{(\bsk,\bsl)}(\nu)\);
            \item\label{item:prime_coefficients_P}
                  For every integer \(r'\) with
                  \(\ell(\bsk,\bsl)\le r'\le n\), the polynomial
                  \(P_{\bsk,\bsl}^{(n,r')}(T)\) appearing in
                  \eqref{eq:differential_operator} has coefficients in
                  \(\bigl(V_{\rho_n}\otimes V_{\rho'_{r'}}\bigr)(\bbZ_{(p)})\).
            \item\label{item:prime_Eisenstein}
                  either \(\mu\ge 4n\), or
                  \(z_{i,p}\in \calO_{K_p}^{\times}\) for all \(1\le i\le h\);
            \item\label{item:prime_cnu}  \(c_\nu\) is \(p\)-unit for all
                  \(\ell(\bsk,\bsl)\le \nu \le n-1\).
          \end{enumerate}
  \end{enumerate}

  Then there exists a Hecke cusp form
  \(F\in \calS^{\sharp}_{n}(\rho_{n})\)
  such that
  \[
    F \equiv_{ev} [f]_{\mu}^{n} \pmod{\frakP}
  \]
  for some prime ideal \(\frakP\) of \(\bbQ(f,F)\) lying above \(\frakp\).

  Moreover, if \(\frakp\) is coprime to \(\frakA(f^{(n)}_d)\) for all \(1\le d \le e_{n}\),
  then there exists a Hecke cusp form
  \(G\in\calS_n^\sharp(\rho_n)\)
  such that
  \[
    G \equiv_{ev} [f]_{\mu}^{n} \pmod{\frakP'^{\alpha}}
  \]
  for some prime ideal \(\frakP'\) of \(\bbQ(f,G)\) lying above \(\frakp\).
\end{theorem}

Before proving Theorem~\ref{thm:main}, we prepare the following lemmata.
\begin{lemma}\label{lem:integral_Klingen_Fourier}
  Assume the same hypotheses as in Theorem~\ref{thm:main}.
  Let \(\ell(\bsk,\bsl)\le \nu \le n-1\) and \(1\le d \le e_{\nu}\), and assume that
  \((\nu,d)\neq (r,1)\).
  Then the Fourier coefficients of
  \[
    \overline{\calC_{2n,\mu}\bigl(f^{(\nu)}_d\bigr)}\,
    \overline{A_{[f^{(\nu)}_{d}]_{\mu}^{n}}(\gamma_0,S_0)}\,
    [f^{(\nu)}_{d}]_{\mu}^{n}(g_1)
  \]
  are \(\frakp\)-integral; that is, they are \(\frakq\)-integral elements of
  \(V_{\rho_{n}}\otimes V_{\rho'_{n}}\bigl(\bbQ(f,f^{(\nu)}_d)\bigr)\)
  for every prime ideal \(\frakq\) of \(\bbQ(f,f^{(\nu)}_d)\) lying above \(\frakp\).
\end{lemma}

\begin{proof}
  For any \(\gamma\in \frakG_0\) and \(S\in \Her_n(K)_{\ge 0}\),
  assumptions~\eqref{item:frakA_condition} and~\eqref{item:prime_eta}
  together with Proposition~\ref{prop:integrality_Klingen} imply that
  \begin{align*}
     & c_\nu \overline{\calC_{2n,\mu}\bigl(f^{(\nu)}_d\bigr)}\,
    A_{[f^{(\nu)}_{d}]_{\mu}^{n}}(\gamma,S)\,
    \overline{A_{[f^{(\nu)}_{d}]_{\mu}^{n}}(\gamma_0,S_0)}      \\
     & \qquad =
    \frac{\bbL_F(2n,\mu)}{\bbL_F(n+\nu,\mu)}\cdot
    \frac{v}{c_\nu\,\calC_{n+\nu,\mu}\bigl(f^{(\nu)}_d\bigr)}   \\
     & \qquad =
    \prod_{i=n+\nu}^{2n-1}L_F(1-\mu+i,\chi_{-D_K}^i)\cdot
    \frac{v}{c_\nu\,\calC_{n+\nu,\mu}\bigl(f^{(\nu)}_d\bigr)}.
  \end{align*}
  where \(v\) is a \(\frakp\)-integral element of
  \(V_{\rho_{n}}\otimes V_{\rho'_{n}}\bigl(\bbQ(f,f^{(\nu)}_d)\bigr)\).

  By Remark~\ref{rem:Bernoulli} and assumptions~\eqref{item:prime_mu}
  and~\eqref{item:prime_DK}, each factor
  \(L_F(1-\mu+i,\chi_{-D_K}^i)\) is \(\frakp\)-integral for
  \(n+\nu \le i \le 2n-1\).
  Therefore, by assumptions~\eqref{item:CalC_condition}
  and~\eqref{item:prime_cnu}, the assertion follows.
\end{proof}

The key to the proof of the first part of the main theorem is the following lemma,
which is proved in the same way as Katsurada~\cite[Lemma~5.1]{katsurada2008congruence}
(see also \cite[Lemma~6.10]{atobe2023harder}).

\begin{lemma}\label{lem:cong}
  Let \(F_1,\ldots,F_d\) be Hecke eigenforms in
  \(\calA^\sharp_{n}(\rho_{n})\) which are linearly independent over
  \(\bbC\).
  Let \(E\) be a number field containing their Hecke fields,
  let \(\calO_E\) be its ring of integers, and let \(\frakQ\) be a prime
  ideal of \(E\).
  Let
  \(G(Z)\in(\calA^\sharp_{n}(\rho_{n})\otimes
  V_{\rho'_{n}})(\calO_{E,\frakQ})\),
  and assume that the following conditions hold:
  \begin{enumerate}
    \item The form \(G\) is expressed as
          \(G(Z)=\sum_{i=1}^d c_i F_i(Z)\)
          with \(c_i \in V_{\rho'_{n}}(E)\).
    \item There exist \(S\in \Her_n(K)_{\ge 0}\) and \(\gamma\in \frakG_0\) such that
          \(c_1 A_{F_1}(\gamma,S)\in
          (V_{\rho_{n}}\otimes V_{\rho'_{n}})(E)\)
          and \(v_\frakQ(c_1 A_{F_1}(\gamma,S))<0\).
  \end{enumerate}
  Then there exists an index \(i\neq 1\) such that
  \[
    F_i \equiv_{ev} F_1 \pmod{\frakP},
  \]
  where \(\frakP\) is a prime ideal of \(\bbQ(F_1,F_i)\) lying above \(\frakQ\).
\end{lemma}

\begin{proof}[Proof of Theorem~\ref{thm:main}]
  Let \(E\) be a finite Galois extension of \(\bbQ\), stable under complex conjugation and containing \(\bbQ(f)\), all fields \(\bbQ(f_d^{(n)})\), and all coefficient fields occurring below.
  Fix a prime ideal \(\frakQ\) of \(E\) whose contraction to \(\bbQ(f)\) is \(\frakp\),
  and write \(e_{\frakQ/\frakp}\) for its ramification index.

  By assumption~\eqref{item:prime_conditions}, we have
  \[
    \calF^{(n,n)}_\mu(g_1,\gamma;S_2)
    \in
    (\calA_{n}^{\sharp}(\rho_{n})\otimes V_{\rho'_{n}})
    (\bbZ_{(p)})
  \]
  for any \(\gamma\in \frakG_0\) and \(S_2\in \Her_n(K)_{\ge 0}\).
  Hence, by Lemma~\ref{lem:integral_Klingen_Fourier},
  \begin{equation}\label{eq:congruence_main}
    c_r\,\overline{\calC_{2n,\mu}(f)}\,
    \overline{A_{[f]_{\mu}^{n}}(\gamma_0,S_0)}\,
    [f]_{\mu}^{n}(g_1)
    +
    \sum_{d=1}^{e_n}\beta_d\, f^{(n)}_{d}(g_1)
    \in
    (\calA_n^\sharp(\rho_n)\otimes V_{\rho'_n})(\calO_{E,\frakQ}),
  \end{equation}
  where
  \[
    \beta_d
    =
    (-1)^{(n-r)\mu/2}\,
    c_n\,
    \overline{\calC_{2n,\mu}(f^{(n)}_{d})}\,
    \overline{A_{[f^{(n)}_{d}]_{\mu}^{n}}(\gamma_0,S_0)}
    \in
    V_{\rho'_{n}}\bigl(\bbQ(f^{(n)}_d)\bigr).
  \]
  As in the proof of Proposition~\ref{prop:integrality_Klingen}, we note that
  \[
    \sum_{d=1}^{e_n}\beta_d\, f^{(n)}_{d}(g_1)
    \in
    (\calA_{n}^{\sharp}(\rho_{n})\otimes V_{\rho'_{n}})(E).
  \]

  By assumptions~\eqref{item:valuation_condition}
  and~\eqref{item:prime_cnu},
  \begin{equation}\label{eq:valuation_negative}
    v_\frakQ\!\left(
    c_r\,\overline{\calC_{2n,\mu}(f)}\,
    A_{[f]_{\mu}^{n}}(\gamma_0,S_0)\,
    \overline{A_{[f]_{\mu}^{n}}(\gamma_0,S_0)}
    \right)
    =-e_{\frakQ/\frakp}\alpha<0.
  \end{equation}
  Hence Lemma~\ref{lem:cong} applies and yields the first assertion.

  We now prove the final assertion, retaining the same field \(E\) and prime
  \(\frakQ\).

  Let \(\{v_1,\ldots,v_t\}\) and \(\{v'_1,\ldots,v'_{t'}\}\) be fixed integral bases of \(V_{\rho_n}\) and \(V_{\rho'_n}\), respectively.
  Write
  \[
    A_{f_d^{(n)}}(\gamma_0,S_0) = \sum_{j=1}^t a_{dj}v_j,
  \]
  and
  \[
    \beta_d = \sum_{k=1}^{t'} b_{dk}v'_k
  \]
  with \(a_{dj},b_{dk}\in E\).

  Since each \(f_d^{(n)}\) is \(p\)-normalized,
  we have \(a_{dj}\in\calO_{E,\frakQ}\) for every \(d\) and \(j\).

  Taking the \((\gamma_0,S_0)\)-th Fourier coefficient in
  \eqref{eq:congruence_main} and using
  \eqref{eq:valuation_negative}, we obtain
  \[
    v_{\frakQ}\left(
    \sum_{d=1}^{e_n}
    A_{f_d^{(n)}}(\gamma_0,S_0)\otimes\beta_d
    \right)
    =
    -e_{\frakQ/\frakp}\alpha.
  \]
  Hence there exist indices \(j\) and \(k\) such that
  \[
    v_{\frakQ}\left(
    \sum_{d=1}^{e_n}a_{dj}b_{dk}
    \right)
    =
    -e_{\frakQ/\frakp}\alpha.
  \]
  After renumbering the two integral bases, we may assume that \(j=k=1\).
  After renumbering the forms \(f_d^{(n)}\),
  we may also assume that, for some \(e_0\ge1\),
  \begin{align*}
    a_{d1}b_{d1} & \ne0
    \quad (1\le d\le e_0), \\
    a_{d1}b_{d1} & =0
    \quad (e_0<d\le e_n).
  \end{align*}
  Thus
  \[
    v_{\frakQ}\left(
    \sum_{d=1}^{e_0}a_{d1}b_{d1}
    \right)
    =
    -e_{\frakQ/\frakp}\alpha.
  \]
  It follows that, after renumbering the first \(e_0\) forms if necessary,
  \begin{equation}\label{eq:valuation_negative_second}
    v_{\frakQ}(a_{11}b_{11})
    \le
    -e_{\frakQ/\frakp}\alpha.
  \end{equation}

  Put \(E_1=\bbQ(f_1^{(n)})\), \(\frakq=\frakQ\cap E_1\), \(\overline{\frakq} = \overline{\frakQ}\cap E_1\).
  Since \(f_1^{(n)}\) is \(p\)-normalized, it has an \(\overline{\frakq}\)-primitive Fourier coefficient.

  For any \(T\in\calH_n^{\rho_n}\),
  applying the operator \(T-\lambda_{[f]_\mu^n}(T)\) to \eqref{eq:congruence_main},
  we obtain
  \[
    H(g_1)
    :=
    \sum_{d=1}^{e_n}
    \beta_d
    \bigl(
    \lambda_{f_d^{(n)}}(T)
    -
    \lambda_{[f]_\mu^n}(T)
    \bigr)
    f_d^{(n)}(g_1)
    \in
    \bigl(
    \calS_n^\sharp(\rho_n)\otimes V_{\rho'_n}
    \bigr)(\calO_{E,\frakQ}).
  \]
  Write \(H=\sum_{k=1}^{t'}H_kv'_k\), with \(H_k\in\calS_n^\sharp(\rho_n)(E)\).

  Since \(a_{11}\in\calO_{E,\frakQ}\), the form \(a_{11}H\) is also integral at \(\frakQ\).
  Hence, by the approximation theorem, there exists a \(\frakQ\)-unit
  \(u\in E^\times\) such that
  \[
    ua_{11}H
    \in
    \bigl(
    \calS_n^\sharp(\rho_n)\otimes V_{\rho'_n}
    \bigr)(\calO_E).
  \]

  Applying Lemma~\ref{lem:integrality} at
  \(\overline{\frakq}\) and \(\overline{\frakQ}\) to
  \(f_1^{(n)}\) and \(ua_{11}H_1\), and using the orthogonality of
  \(f_1^{(n)},\ldots,f_{e_n}^{(n)}\), we obtain
  \begin{align*}
    \frac{
    \bigl(f_1^{(n)},ua_{11}H_1\bigr)
    }{
    \bigl(f_1^{(n)},f_1^{(n)}\bigr)
    }
     & =
    \overline{
    ua_{11}b_{11}
    \bigl(
    \lambda_{f_1^{(n)}}(T)
    -
    \lambda_{[f]_\mu^n}(T)
    \bigr)
    }      \\
     & \in
    \frakA(f_1^{(n)})^{-1}
    \calO_{E,\overline{\frakQ}}.
  \end{align*}

  Taking complex conjugates and using the invariance of
  \(\frakA(f_1^{(n)})\) under complex conjugation, we obtain
  \[
    ua_{11}b_{11}
    \bigl(
    \lambda_{f_1^{(n)}}(T)
    -
    \lambda_{[f]_\mu^n}(T)
    \bigr)
    \in
    \frakA(f_1^{(n)})^{-1}
    \calO_{E,\frakQ}.
  \]
  Since \(u\) is a \(\frakQ\)-unit and
  \(\frakA(f_1^{(n)})\) is prime to \(\frakQ\),
  \eqref{eq:valuation_negative_second} gives
  \[
    v_{\frakQ}\left(
    \lambda_{f_1^{(n)}}(T)
    -
    \lambda_{[f]_\mu^n}(T)
    \right)
    \ge
    e_{\frakQ/\frakp}\alpha.
  \]
  Consequently,
  \[
    v_{\frakP'}\left(
    \lambda_G(T)-\lambda_{[f]^n_\mu}(T)
    \right)
    \ge e_{\frakP'/\frakp}\alpha
    \ge\alpha
  \]
  for every \(T\in\calH_n^{\rho_n}\).
  Thus
  \(G\equiv_{ev}[f]^n_\mu\pmod{\frakP'^{\alpha}}\), which is the
  final assertion.
  This completes the proof.
\end{proof}

The same argument, without separating the cuspidal summand, gives the following
weaker conclusion under fewer assumptions.
\begin{corollary}\label{cor:main}
  Under the same notation as in Theorem~\ref{thm:main},
  assume that the conditions~\eqref{item:valuation_condition}, \eqref{item:prime_mu},
  \eqref{item:prime_DK}, \eqref{item:prime_coefficients_P},
  \eqref{item:prime_Eisenstein}, and \eqref{item:prime_cnu}
  are satisfied.
  Then there exists a Hecke eigenform
  \(F \in \calA^{\sharp}_{n}(\rho_{n})\)
  such that
  \[
    F \equiv_{ev} [f]_{\mu}^{n} \pmod{\frakP}
  \]
  for some prime ideal \(\frakP\) of \(\bbQ(f,F)\) lying above \(\frakp\).
\end{corollary}

\section{Example}\label{sec:example}

We treat the case \(K=\bbQ(\sqrt{-3})\), whose discriminant is \(D_K=3\).
For each \(i \in \{12,22\}\), let
\(\Delta_i \in S_i(\SL_2(\bbZ)) = S_i(\Gamma_K^{(1)})\)
be the unique cusp form of weight \(i\) and level \(1\), normalized so that the Fourier coefficient of \(q\) is \(1\).

A direct calculation gives
\begin{align*}
  F_3(X; I_3)            & = 1 + 3^3 X,                        \\
  F_3(X; I_4)            & = 1 + 2\cdot 3^5 X + 3^8 X^2,       \\
  F_3(X; \diag(1,3))     & = 1 - 3^4 X^2,                      \\
  F_3(X; \diag(1,1,3))   & = 1 - 3^3(3-1) X + 3^6 X^2,         \\
  F_3(X; \diag(1,1,1,3)) & = 1 + 3^5 X + 3^9 X^2 + 3^{12} X^3.
\end{align*}

We briefly indicate the computation.
By the definition of the local Hermitian
Siegel series and its normalizing factor \cite[p.~1112]{ikeda2008lifting},
the coefficient of \(X\) is determined by the terms of degree at most one in
the defining local sum.
Since \(K_3/\bbQ_3\) is ramified, the corresponding
finite exponential sums reduce to elementary quadratic Gauss sums over
\(\calO_{K_3}/\sqrt{-3}\calO_{K_3}\simeq\mathbb F_3\).
The degrees and the remaining coefficients are then determined by the
functional equation \cite[Corollary~3.2]{ikeda2008lifting}.

\noindent
\textbf{(1)} \(\mu=8\), \(r=1\), \(n=2\), \(\bsk=(7,0)\), \(\bsl=(7,0)\).

We identify the representation space of \(\det^4\Sym^7\boxtimes\det^4\Sym^7\) with the space of homogeneous polynomials of degree \(7\) in the variables \(u_1,u_2\) and of degree \(7\) in the variables \(v_1,v_2\).
Since \(\dim \calS_{1}(\det^{11}\boxtimes\det^{11}) = 1\), we have \(\frakA(\Delta_{22}) = 1\).
By direct computation, we have
\[
  c_1 = \frac{2^{15} \cdot 3^3}{11^2\cdot13^2}.
\]
Using SageMath, we compute the constant \(\calC_{4,8}(\Delta_{22})\) as follows:
\begin{align*}
  \calC_{4,8}(\Delta_{22})
   & = \frac{\bbL_F(4, 8)}{\bbL_F(2, 8)} \bbL(7/2, \Delta_{22}) \\
   & = \zeta(1-6) L(1-5, \chi_{-3}) \Gamma_\bbC(14)^2
  \frac{L(14, \Delta_{22}) L(14, \Delta_{22} \otimes \chi_{-3})}
  {\sqrt{3} (\Delta_{22}, \Delta_{22})}                         \\
   & = -\frac{2^{24} \cdot 41}{3^{15}\cdot5 \cdot 19}.
\end{align*}

Furthermore, by \cite[Theorem~5.13]{takeda2025differential}, the differential operator \(P_{(7,0),(7,0)}(T)\) for \(T=(t_{ij})\) is explicitly given by
\[
  P_{(7,0),(7,0)}(T)
  = \frac{1}{2^{16}\cdot3^8\cdot5^3\cdot7^2\cdot11^2\cdot13^2\cdot17\cdot19}
  \sum_{j=0}^7 (-1)^j \binom{7}{j}\binom{13}{7-j}
  A^{7-j}B^{7-j}C^j,
\]
where
\begin{align*}
  A & = u_1t_{13}+u_2t_{23},     \\
  B & = v_1t_{31}+v_2t_{32},     \\
  C & = \sum_{i=1}^2\sum_{j=1}^2
  u_iv_j(t_{ij}t_{33}-t_{i3}t_{3j}).
\end{align*}
By substituting this explicit form of \(P_{(7,0),(7,0)}(T)\) into the formula in Proposition~\ref{prop:Eisen_Fourier}, we evaluate the Fourier coefficient of \([\Delta_{22}]_{8}^{2}\) at \(T_1 = \begin{pmatrix} 2 & 1 \\ 1 & 2 \end{pmatrix}\).
A direct computation then yields
\begin{align*}
      & c_1\mathcal{C}_{4,8}(\Delta_{22})
  A_{[\Delta_{22}]_{8}^{2}}(I_2;T_1)                           \\
  ={} & -L(1-5,\chi_{-3})\epsilon_{8}(I_2,I_1;T_1,1)           \\
  ={} & \frac{1}{N}\bigg[
  9644504873600 \bigl(u_1^7v_1^7+u_2^7v_2^7\bigr)              \\
      & \quad +33755767057600 \bigl(
    u_1^7v_1^6v_2+u_1^6u_2v_1^7
  +u_1u_2^6v_2^7+u_2^7v_1v_2^6\bigr)                           \\
      & \quad +49909941436320 \bigl(
    u_1^7v_1^5v_2^2+u_1^5u_2^2v_1^7
  +u_1^2u_2^5v_2^7+u_2^7v_1^2v_2^5\bigr)                       \\
      & \quad +40385435946800 \bigl(
    u_1^7v_1^4v_2^3+u_1^4u_2^3v_1^7
  +u_1^3u_2^4v_2^7+u_2^7v_1^3v_2^4\bigr)                       \\
      & \quad +18067650417880 \bigl(
    u_1^7v_1^3v_2^4+u_1^4u_2^3v_2^7
  +u_1^3u_2^4v_1^7+u_2^7v_1^4v_2^3\bigr)                       \\
      & \quad +3593923208820 \bigl(
    u_1^7v_1^2v_2^5+u_1^5u_2^2v_2^7
  +u_1^2u_2^5v_1^7+u_2^7v_1^5v_2^2\bigr)                       \\
      & \quad -4203609285722 \bigl(
    u_1^7v_1v_2^6+u_1^6u_2v_2^7
  +u_1u_2^6v_1^7+u_2^7v_1^6v_2\bigr)                           \\
      & \quad -2158319385571 \bigl(u_1^7v_2^7+u_2^7v_1^7\bigr) \\
      & \quad +119984667172160 \bigl(
  u_1^6u_2v_1^6v_2+u_1u_2^6v_1v_2^6\bigr)                      \\
      & \quad +180203242438800 \bigl(
    u_1^6u_2v_1^5v_2^2+u_1^5u_2^2v_1^6v_2
  +u_1^2u_2^5v_1v_2^6+u_1u_2^6v_1^2v_2^5\bigr)                 \\
      & \quad +137136892510240 \bigl(
    u_1^6u_2v_1^4v_2^3+u_1^4u_2^3v_1^6v_2
  +u_1^3u_2^4v_1v_2^6+u_1u_2^6v_1^3v_2^4\bigr)                 \\
      & \quad +45615097502660 \bigl(
    u_1^6u_2v_1^3v_2^4+u_1^4u_2^3v_1v_2^6
  +u_1^3u_2^4v_1^6v_2+u_1u_2^6v_1^4v_2^3\bigr)                 \\
      & \quad +25067803385772 \bigl(
    u_1^6u_2v_1^2v_2^5+u_1^5u_2^2v_1v_2^6
  +u_1^2u_2^5v_1^6v_2+u_1u_2^6v_1^5v_2^2\bigr)                 \\
      & \quad +7507020370115 \bigl(
  u_1^6u_2v_1v_2^6+u_1u_2^6v_1^6v_2\bigr)                      \\
      & \quad +316319978807280 \bigl(
  u_1^5u_2^2v_1^5v_2^2+u_1^2u_2^5v_1^2v_2^5\bigr)              \\
      & \quad +300063203952120 \bigl(
    u_1^5u_2^2v_1^4v_2^3+u_1^4u_2^3v_1^5v_2^2
  +u_1^3u_2^4v_1^2v_2^5+u_1^2u_2^5v_1^3v_2^4\bigr)             \\
      & \quad +65077833536730 \bigl(
    u_1^5u_2^2v_1^3v_2^4+u_1^4u_2^3v_1^2v_2^5
  +u_1^3u_2^4v_1^5v_2^2+u_1^2u_2^5v_1^4v_2^3\bigr)             \\
      & \quad -10975684259799 \bigl(
  u_1^5u_2^2v_1^2v_2^5+u_1^2u_2^5v_1^5v_2^2\bigr)              \\
      & \quad +520560554322760 \bigl(
  u_1^4u_2^3v_1^4v_2^3+u_1^3u_2^4v_1^3v_2^4\bigr)              \\
      & \quad +342194119803095 \bigl(
    u_1^4u_2^3v_1^3v_2^4+u_1^3u_2^4v_1^4v_2^3\bigr)
    \bigg],
\end{align*}
where \(N=2^{11}\cdot3^{12}\cdot5^3\cdot7^2\cdot11^2\cdot13^2\cdot17\cdot19\).

Consequently, for a prime \(p=41\), we have
\[
  v_{41}\!\left(
  \overline{\calC_{4,22}(\Delta_{22})}\,
  A_{[\Delta_{22}]_{8}^{2}}(I_2, T_1) \cdot
  \overline{A_{[\Delta_{22}]_{8}^{2}}(I_2, T_1)}
  \right) = -1.
\]
Since these primes satisfy the hypotheses of Theorem~\ref{thm:main}, we conclude that for each such \(p\), there exists a Hecke cusp form \(F \in \calS^{\sharp}_{2}\bigl(\det^{4}\Sym^7\boxtimes \det^{4}\Sym^7\bigr)\) such that
\[
  F \equiv [\Delta_{22}]_{8}^{2} \pmod{\frakP}
\]
for some prime ideal \(\frakP\) of the coefficient field \(\bbQ(F)\) lying above \(41\).

\noindent
\textbf{(2)} \(\mu=12\), \(r=1\), \(n=3\), \(\bsk=\bsl=0\).

This is the scalar-valued case, and the computation proceeds similarly:
\begin{align*}
  c_1 & =-2^9\cdot3^{10},                    \\
  \calC_{6,12}(\Delta_{12})
      & =\frac{2^{17}\cdot5\cdot7^2\cdot809}
  {3^{14}\cdot11\cdot691}.
\end{align*}
Again by Proposition~\ref{prop:Eisen_Fourier},
\begin{align*}
  c_1\calC_{6,12}(\Delta_{12})
  A_{[\Delta_{12}]_{12}^{3}}(I_3;T_2)
   & = \zeta(1-8)L(1-7,\chi_{-3})
  \epsilon_{12}(I_3,I_1;T_2,1)                 \\
   & = -2^2\cdot3\cdot7\cdot23\cdot149\cdot421
  \cdot42491\cdot6510393709,
\end{align*}
where \(T_2=\begin{pmatrix}2&1&0\\1&2&0\\0&0&1\end{pmatrix}\).

We then obtain
\[
  v_{809}\!\left(
  \overline{\calC_{6,12}(\Delta_{12})}\,
  A_{[\Delta_{12}]_{12}^{3}}(I_3;T_2)
  \overline{A_{[\Delta_{12}]_{12}^{3}}(I_3;T_2)}
  \right)=-1.
\]
From the above computations together with some additional straightforward verifications,
we conclude that \(809\) satisfies the assumptions of Corollary~\ref{cor:main}.
Hence there exists a Hecke eigenform
\(
F \in \calA^{\sharp}_{3}(\det^6\boxtimes\det^6)
\)
such that
\[
  F \equiv_{ev} [\Delta_{12}]_{12}^{3} \pmod{\frakP}
\]
for some prime ideal \(\frakP\) of \(\bbQ(F)\) lying above \(809\).

\bibliography{Hermitian_congruence}
\bibliographystyle{amsplain}

\end{document}